\documentclass[11pt]{article}
\usepackage[T1]{fontenc}
\usepackage[utf8]{inputenc}
\usepackage{lmodern,microtype}
\usepackage[margin=1in]{geometry}
\usepackage{amsmath,amssymb,amsthm,mathtools}
\usepackage{tikz}
\PassOptionsToPackage{hyphens}{url}
\usepackage[colorlinks=true,linkcolor=blue,citecolor=blue,urlcolor=blue]{hyperref}

\numberwithin{equation}{section}
\numberwithin{figure}{section}
\numberwithin{table}{section}
\theoremstyle{plain}
\newtheorem{theorem}{Theorem}[section]
\newtheorem{proposition}[theorem]{Proposition}
\newtheorem{lemma}[theorem]{Lemma}
\newtheorem{corollary}[theorem]{Corollary}
\theoremstyle{definition}
\newtheorem{hypothesis}[theorem]{Hypothesis}
\newtheorem{definition}[theorem]{Definition}
\theoremstyle{remark}
\newtheorem{remark}[theorem]{Remark}

\newcommand{\R}{\mathbb R}
\newcommand{\D}{\mathbb D}
\newcommand{\Horn}{\mathcal G} % the Gaussian horn (Section 5)
\newcommand{\Om}{\Omega}
\newcommand{\norm}[2]{\left\lVert #1\right\rVert_{#2}}
\DeclareMathOperator{\dist}{dist}

\title{Hot Spots: From Smooth to Nonsmooth \\Simply Connected Planar Domains}
\author{Rodrigo Ba\~{n}uelos}
\date{October 20, 2026}

\begin{document}
\maketitle

\begin{abstract}

Assuming the absence of interior local extrema for first positive Neumann eigenfunctions on bounded smooth simply connected planar domains, as reported in \cite{SmoothInput}, we extend the hot spots conclusion to simply connected domains of finite area with compact Sobolev embedding and simple first positive Neumann eigenvalue. The supremum and infimum are approached in arbitrarily thin boundary layers. For bounded domains and eigenfunctions continuous up to the boundary, both extrema are attained on the boundary. We also transfer polynomial Neumann heat-kernel bounds uniformly along a conformal exhaustion and construct explicit bounded and unbounded examples, including domains with cusps and a domain with infinitely many corners and bottlenecks. Some of these examples admit no polynomial Neumann heat-kernel bound. 
\end{abstract}

\section*{AI Disclosure}
The initial draft of Section~\ref{sec:main}, including the proof extending the smooth-domain assertion reported in \cite{SmoothInput} to bounded nonsmooth simply connected domains, was generated entirely by ChatGPT (6.1 Sol) in response to a few human prompts. Independently and concurrently, K.~Burdzy used ChatGPT to obtain an almost identical theorem and proof, as reported in \cite[Human Commentary, item (ii)]{BurdzyNote}.

The initial drafts of Sections~\ref{sec:HK}--\ref{sec:withoutSimplicity}, which develop the heat-kernel approach, the nonsmooth examples, the unbounded examples, and the discussion of simplicity, were also generated by ChatGPT (6.1 Sol), with substantial input, guidance, and scrutiny from the author. The present version incorporates a further AI-assisted review and revision of the arguments, references, and exposition. The author's main interest was to identify additional examples of domains with nonsmooth boundaries and compact Neumann resolvent, including examples satisfying polynomial heat-kernel bounds, for which the first positive Neumann eigenvalue is simple. The aim was to apply the result in Section~\ref{sec:main} to these domains. The mathematical ideas in these sections, including the introduction of heat-kernel bounds and cross-cut domains and techniques, are due to the author. The author reviewed and edited all AI-assisted material and checked the arguments. 
The smooth-domain assertion enters only through the explicit Hypothesis~\ref{hyp:smooth}. Its proof is not assessed here.

The author does not intend to submit these notes for journal publication. 

\clearpage
\setcounter{tocdepth}{2}
\tableofcontents
\clearpage
\section*{Introduction}
\phantomsection
\addcontentsline{toc}{section}{Introduction}

The hot spots conjecture of J.~Rauch asserts that a first positive Neumann eigenfunction of a planar domain attains its maximum and minimum on the boundary. See \cite{BB} for a precise formulation. The conjecture has inspired a vast literature, with connections to probability, analysis, partial differential equations, and geometry. These notes make no attempt to survey this literature, and results closely related to the questions considered here may have been overlooked. These notes take an assertion about smooth domains as an external hypothesis and derive a weaker conclusion for domains with irregular boundary.

\paragraph{The hypothesis.}
The external input, reported in \cite{SmoothInput} and stated as Hypothesis~\ref{hyp:smooth}, is that on a bounded smooth simply connected planar domain no first positive Neumann eigenfunction has an interior local maximum or minimum. The assertion in \cite{SmoothInput} is stronger, namely that such eigenfunctions have no interior critical points. Its proof is not examined here.

\paragraph{The main result.}
Let $\Om\subset\R^2$ be a simply connected domain of finite area, and suppose that the embedding $H^1(\Om)\hookrightarrow L^2(\Om)$ is compact and that the first positive Neumann eigenvalue $\lambda_1(\Om)$ is simple. Theorem~\ref{thm:main} states that, under Hypothesis~\ref{hyp:smooth}, every real eigenfunction $u$ for $\lambda_1(\Om)$ satisfies
\[
 \sup_\Om u=\lim_{\delta\downarrow0}\ \sup\{u(x):x\in\Om,\ \dist(x,\partial\Om)<\delta\},
\]
and the same identity for the infimum. In words: the supremum and the infimum are approached in arbitrarily thin boundary layers. If $\Om$ is bounded and $u$ is continuous on $\overline\Om$, this says that the global extrema of $u$ are attained on the boundary. The conclusion is nonstrict. It does not exclude an interior point at which an extremal value is also attained. No regularity of $\partial\Om$ is assumed.

The proof is by conformal exhaustion. Given a conformal bijection $F:\D\to\Om$, the domains $\Om_r=F(r\D)$, $0<r<1$, are bounded by analytic Jordan curves and increase to $\Om$. Reflection in the parameter disk extends functions from $\Om_r$ to $\Om$ with at most twice the energy. Compact embedding then gives convergence of the first positive eigenvalues and locally uniform convergence of the eigenfunctions. If $u$ had a compact superlevel set in $\Om$ above its boundary-layer value, the approximating eigenfunctions would have interior local maxima on the $\Om_r$, contradicting Hypothesis~\ref{hyp:smooth}. Burdzy's note \cite{BurdzyNote}, obtained independently, uses the same mechanism. It is stated for bounded domains but the proof is essentially the same for domains of finite area. 

Two recent preprints give further developments on the hot spots
problem. Deng, Jiang, and Yang \cite{DengJiangYang2026} prove that
the maximum modulus of every first positive Neumann eigenfunction
on a bounded convex planar domain is attained only on the boundary.
Their result requires neither boundary smoothness nor simplicity
of the eigenvalue. For centrally symmetric convex planar domains,
they obtain the strict hot spots conclusion for both extrema.
Atar and Burdzy \cite{AtarBurdzy2026} construct a bounded starlike
Lipschitz domain in three dimensions with simple first positive
Neumann eigenvalue whose corresponding eigenfunction has its
unique global maximum at an interior point.

\paragraph{Two routes, and the role of heat-kernel bounds.}
Section~\ref{sec:main} proves Theorem~\ref{thm:main} under compact embedding alone. Section~\ref{sec:HK} concerns the stronger assumption of a polynomial Neumann heat-kernel bound $p_\Om(t,x,y)\le Ct^{-N/2}$ for $0<t\le1$, where $N>2$ is an exponent, not the dimension. Such a bound implies compact embedding, so the hot-spots conclusion for these domains is a special case of Theorem~\ref{thm:main}. Section~\ref{sec:HK} is included for Proposition~\ref{prop:HK}, which shows that a bound on the target transfers to every $\Om_r$ with $r$ close to $1$, with the same exponent and a constant independent of $r$. This is the input required by the stability theory of Ba\~{n}uelos--Pang \cite{BP}, and it gives a second proof of the hot-spots conclusion for this class.

\paragraph{Organization.}
Section~\ref{sec:examples} lists classes of domains with polynomial heat-kernel bounds: Lipschitz domains, quasidisks, H\"older graph domains, and power cusps. Section~\ref{sec:simplicity} compares the Neumann eigenvalues of a domain lying between two graphs with those of a weighted one-dimensional problem obtained by averaging over vertical cross-cuts. This gives explicit simplicity criteria and explicit examples: Lipschitz tubes, quadratic cusps, and an exponential cusp that has compact embedding and a simple first positive eigenvalue but no polynomial heat-kernel bound. Section~\ref{sec:unboundedFiniteArea} treats two unbounded examples of finite area: the Gaussian horn and a rooms-and-passages domain with infinitely many bottlenecks and corners. Section~\ref{sec:withoutSimplicity} shows that, without simplicity, the exhaustion argument still yields at least one first positive eigenfunction satisfying the boundary-layer conclusion. Appendix~\ref{app:classes} collects further classes with polynomial heat-kernel bounds (quasihyperbolic H\"older domains, finite decompositions, a slit disk) that are not needed for the main line.

\section{Hot spots under compact Sobolev embedding}\label{sec:main}

\subsection{Setting and statements}
Throughout, Sobolev spaces and variational arguments are taken over the real scalars. A \emph{domain} is a nonempty connected open set, and $H^1(\Om)$ is the Sobolev space of square-integrable functions on $\Om$ with square-integrable weak first derivatives. For a domain $\Om\subset\R^2$ of finite area, let $A_\Om$ be the nonnegative self-adjoint Neumann Laplacian associated with the closed form
\[
 \mathcal E_\Om(v,w)=\int_\Om\nabla v\cdot\nabla w,
 \qquad \mathcal D(\mathcal E_\Om)=H^1(\Om).
\]
Its semigroup is $T_t^\Om=e^{-tA_\Om}$, so the associated heat equation is $\partial_tv=\Delta v$ (Brownian motion with generator $\tfrac12\Delta$ corresponds to a change of time). This definition requires neither a trace nor a pointwise normal derivative on $\partial\Om$. We write $\lambda_0(\Om)\le\lambda_1(\Om)\le\lambda_2(\Om)\le\cdots$ for the min--max values of $\mathcal E_\Om$ relative to the $L^2(\Om)$ norm,
\[
 \lambda_k(\Om)=\inf_{W}\ \sup_{w\in W\setminus\{0\}}
 \frac{\int_\Om|\nabla w|^2}{\int_\Om w^2},
\]
the infimum running over all $(k+1)$-dimensional subspaces $W\subset H^1(\Om)$. Compactness of the embedding $H^1(\Om)\hookrightarrow L^2(\Om)$ means that every bounded sequence in $H^1(\Om)$ has an $L^2(\Om)$-convergent subsequence. When the embedding is compact, the $\lambda_k(\Om)$ are the eigenvalues of $A_\Om$ counted with multiplicity; $\lambda_0(\Om)=0$ because constants lie in $L^2(\Om)$, $\lambda_1(\Om)>0$ by connectedness, finite area, and compact embedding (Lemma~\ref{lem:anchor}), and $\lambda_1(\Om)$ is the \emph{first positive eigenvalue}. It is \emph{simple} if its eigenspace is one-dimensional, equivalently if $\lambda_1(\Om)<\lambda_2(\Om)$. A \emph{first positive Neumann eigenfunction} is a nonzero real $u\in H^1(\Om)$ with
\[
 \int_\Om\nabla u\cdot\nabla\varphi=\lambda_1(\Om)\int_\Om u\varphi
 \qquad\text{for all }\varphi\in H^1(\Om).
\]
By interior elliptic regularity it is smooth in $\Om$, and we always use this smooth representative. For $\delta>0$, the \emph{boundary layer} of $\Om$ is
\[
 B_\delta(\Om)=\{x\in\Om:\dist(x,\partial\Om)<\delta\}.
\]

\begin{hypothesis}[Smooth-domain input]\label{hyp:smooth}
For every bounded smooth simply connected planar domain $G$, every nonzero real eigenfunction for its first positive Neumann eigenvalue has no interior local maximum or minimum.
\end{hypothesis}

Here \emph{smooth} may be read as $C^\infty$. The hypothesis is applied below only to domains bounded by analytic Jordan curves. The assertion reported in \cite{SmoothInput} is that every first positive Neumann eigenfunction on a bounded smooth simply connected planar domain has no interior critical points, which implies Hypothesis~\ref{hyp:smooth}. The source is a model-generated preprint, accompanied by a Lean statement of the result, in a collection whose review status is listed as unchecked. Neither the proof nor the Lean development is assessed here. The weaker statement that global extrema are attained on the boundary would not suffice for the argument below.

\begin{theorem}\label{thm:main}
Assume Hypothesis~\ref{hyp:smooth}. Let $\Om\subset\R^2$ be a simply connected domain with $|\Om|<\infty$, and suppose that
\begin{equation}\label{eq:CI}
 H^1(\Om)\hookrightarrow L^2(\Om)\quad\text{is compact},
\end{equation}
and that its first positive Neumann eigenvalue $\lambda_1(\Om)$ is simple. Then every real eigenfunction $u$ for $\lambda_1(\Om)$ satisfies
\begin{equation}\label{eq:HS}
 \sup_\Om u=\lim_{\delta\downarrow0}\sup_{B_\delta(\Om)}u,
 \qquad
 \inf_\Om u=\lim_{\delta\downarrow0}\inf_{B_\delta(\Om)}u,
\end{equation}
where infinite extremal values are allowed. If $\Om$ is bounded and $u\in C(\overline\Om)$, then
\begin{equation}\label{eq:boundary}
 \max_{\overline\Om}u=\max_{\partial\Om}u,
 \qquad \min_{\overline\Om}u=\min_{\partial\Om}u.
\end{equation}
\end{theorem}

\begin{remark}[Scope of the conclusion]\label{rem:scopeMain}
No regularity of $\partial\Om$, no boundary continuity of the conformal map, and no boundedness of $u$ is assumed in \eqref{eq:HS}. The conclusion is nonstrict. It does not exclude interior points sharing an extremal value. On an unbounded domain, continuity of $u$ on $\overline\Om$ does not ensure that an extremum is attained, and \eqref{eq:HS} asserts no attainment at a boundary point. If $\Om$ is unbounded and thin at infinity, the layer $B_\delta(\Om)$ contains a whole end of $\Om$ for every $\delta>0$, so one of the identities in \eqref{eq:HS} holds for trivial reasons whenever the corresponding extremum is approached along that end. See Remark~\ref{rem:hornVacuous} for the Gaussian horn.
\end{remark}

\subsection{Compactness and the extension operator}
Fix a conformal bijection $F:\D\to\Om$, where $\D=\{z\in\mathbb C:|z|<1\}$, and put
\[
 \Om_r=F(r\D)\qquad(0<r<1).
\]
Each $\Om_r$ is bounded by an analytic Jordan curve, $\Om_r\Subset\Om$, and $\Om_r\uparrow\Om$ as $r\uparrow1$. Every compact subset of $\Om$ lies in $\Om_r$ for all sufficiently large $r$, and $\overline{\Om_{1/2}}\subset\Om_r$ whenever $r>1/2$. The domain $\Om_{1/2}$ serves as an anchor in the Poincar\'e inequality below.

\begin{lemma}[Anchored Poincar\'e inequality]\label{lem:anchor}
Assume $|\Om|<\infty$ and \eqref{eq:CI}. Then there is $C_1<\infty$ such that
\begin{equation}\label{eq:anchor}
 \norm{w}{L^2(\Om)}^2\le C_1\left(
 \norm{\nabla w}{L^2(\Om)}^2+\norm{w}{L^2(\Om_{1/2})}^2\right)
 \qquad(w\in H^1(\Om)).
\end{equation}
Moreover,
\begin{equation}\label{eq:variational}
 \lambda_1(\Om)=\min\left\{\int_\Om|\nabla w|^2:
 w\in H^1(\Om),\ \int_\Om w=0,\ \int_\Om w^2=1\right\}>0.
\end{equation}
\end{lemma}
\begin{proof}
Suppose \eqref{eq:anchor} fails. Then there are $w_j\in H^1(\Om)$ with $\norm{w_j}{L^2(\Om)}=1$ and $\norm{\nabla w_j}{L^2(\Om)}+\norm{w_j}{L^2(\Om_{1/2})}\to0$. By compactness, a subsequence converges strongly in $L^2(\Om)$ to some $w$ with $\norm{w}{L^2(\Om)}=1$. Its weak gradient vanishes, so $w$ is constant because $\Om$ is connected, while $w=0$ on $\Om_{1/2}$. Hence $w=0$, a contradiction. This part uses only connectedness and compactness.

Finite area enters in \eqref{eq:variational}: constants belong to $L^2(\Om)$, the mean $\int_\Om w$ is defined for $w\in L^2(\Om)$ by the Cauchy--Schwarz inequality, and $|\Om|^{-1}\int_\Om h$ is finite for $h\in H^1(\Om)$. The argument of the first paragraph, with the constraint $\int_\Om w=0$ in place of the anchor, gives a positive Poincar\'e constant. The direct method, which uses compactness of the embedding and weak lower semicontinuity of the Dirichlet integral, then yields a minimizer in \eqref{eq:variational}. The Euler--Lagrange equation holds for mean-zero variations. Applying it to $h-|\Om|^{-1}\int_\Om h$ for an arbitrary test function $h\in H^1(\Om)$ gives the full weak Neumann equation, with $\lambda_1(\Om)$ the minimal value.
\end{proof}

\begin{lemma}[Inversion extension]\label{lem:extension}
Let $|\Om|<\infty$. For $0<r<1$ and $v\in C^\infty(\overline{\Om_r})$, there is a bounded function $\Phi_rv\in H^1(\Om)$, equal to $v$ on $\Om_r$, such that
\begin{equation}\label{eq:energy}
 \int_\Om|\nabla \Phi_rv|^2\le2\int_{\Om_r}|\nabla v|^2.
\end{equation}
This statement does not use \eqref{eq:CI}. If \eqref{eq:CI} holds and $r>1/2$, then, with a constant $C$ independent of $r$ and $v$,
\begin{equation}\label{eq:extensionH}
 \norm{\Phi_rv}{H^1(\Om)}^2\le C\left(
 \int_{\Om_r}|\nabla v|^2+\int_{\Om_{1/2}}|v|^2\right).
\end{equation}
\end{lemma}
\begin{proof}
Write $g=v\circ F$ on $\overline{r\D}$. It is smooth there because $F$ is holomorphic with nonvanishing derivative in a neighborhood of $\overline{r\D}$. Define
\[
 \widetilde g(z)=
 \begin{cases}
 g(z),&|z|\le r,\\
 g(I_r(z)),&r<|z|<1,
 \end{cases}
 \qquad I_r(z)=\dfrac{r^2}{\overline z}.
\]
The two definitions agree on $|z|=r$. The map $I_r$ is anticonformal, maps $\{r<|z|<1\}$ bijectively onto $\{r^2<|z|<r\}$, and scales lengths at $z$ by the factor $r^2/|z|^2\le1$. Hence $\widetilde g$ is continuous on $\D$ and piecewise smooth with bounded gradient, so $\widetilde g\in W^{1,\infty}(\D)\subset H^1(\D)$. No singular part of the gradient arises on the circle $|z|=r$. Since the planar Dirichlet integral is invariant under conformal and anticonformal maps,
\[
 \int_{\D\setminus r\D}|\nabla\widetilde g|^2
 =\int_{\{r^2<|z|<r\}}|\nabla g|^2
 \le\int_{r\D}|\nabla g|^2.
\]
Set $\Phi_rv=\widetilde g\circ F^{-1}$. By the chain rule this function belongs to $H^1_{\mathrm{loc}}(\Om)$, and a further change of variables gives \eqref{eq:energy}. It is bounded by $\norm{v}{L^\infty(\Om_r)}$, hence lies in $L^2(\Om)$ since $|\Om|<\infty$. Its weak derivatives have finite global $L^2$ norm by \eqref{eq:energy}, so $\Phi_rv\in H^1(\Om)$.

Now assume \eqref{eq:CI} and $r>1/2$. Applying \eqref{eq:anchor} to $\Phi_rv$, and using $\Phi_rv=v$ on $\Om_{1/2}\subset\Om_r$ together with \eqref{eq:energy}, gives
\[
 \norm{\Phi_rv}{H^1(\Om)}^2\le2(1+C_1)\int_{\Om_r}|\nabla v|^2+C_1\int_{\Om_{1/2}}|v|^2,
\]
which is \eqref{eq:extensionH}. By density on the smooth domain $\Om_r$, this linear extension map extends continuously to all of $H^1(\Om_r)$ when \eqref{eq:CI} holds and $r>1/2$. The extended map still agrees with the original function on $\Om_r$.
\end{proof}

\subsection{Spectral convergence under compact embedding}
A first positive eigenfunction $u$ on a domain of finite area is \emph{normalized} if $\int u=0$ and $\int u^2=1$. The first condition is automatic, since eigenfunctions are orthogonal to constants.

\begin{proposition}\label{prop:spectral}
Let $\Om$ be simply connected with $|\Om|<\infty$ and satisfy \eqref{eq:CI}. Then
\begin{equation}\label{eq:eigenconv}
 \lambda_1(\Om_r)\longrightarrow\lambda_1(\Om)\qquad(r\uparrow1).
\end{equation}
For every sequence $r_j\uparrow1$ and every choice of normalized real first positive Neumann eigenfunctions $u_j$ on $\Om_{r_j}$, there are a subsequence and a normalized first positive Neumann eigenfunction $v$ on $\Om$ such that $u_j\to v$ in $C^\infty_{\mathrm{loc}}(\Om)$ along the subsequence. If $\lambda_1(\Om)$ is simple, the signs of the $u_j$ can be chosen so that the limit is any prescribed normalized eigenfunction $u$.
\end{proposition}
\begin{proof}
Put $\lambda=\lambda_1(\Om)$ and $\lambda_r=\lambda_1(\Om_r)$, and let $u$ be a normalized first positive eigenfunction on $\Om$.

\emph{Upper bound.} Set $c_r=|\Om_r|^{-1}\int_{\Om_r}u$. Since $u\in L^2(\Om)\subset L^1(\Om)$ by finite area, $\int_{\Om_r}u\to\int_\Om u=0$ and $|\Om_r|\uparrow|\Om|$ as $r\uparrow1$. Hence $c_r\to0$, and by monotone convergence
\[
 \int_{\Om_r}(u-c_r)^2\to1,\qquad
 \int_{\Om_r}|\nabla u|^2\to\lambda.
\]
Since $u-c_r$ has mean zero on $\Om_r$ and is nonzero for large $r$, it is an admissible trial function on $\Om_r$, and the variational principle yields
\begin{equation}\label{eq:limsup}
 \limsup_{r\uparrow1}\lambda_r\le\lambda.
\end{equation}

\emph{Compact extensions.} Choose $r_j\uparrow1$ with $r_j>1/2$ and normalized first positive eigenfunctions $u_j$ on $\Om_{r_j}$, so that
\[
 \int_{\Om_{r_j}}u_j=0,\qquad
 \int_{\Om_{r_j}}u_j^2=1,\qquad
 \int_{\Om_{r_j}}|\nabla u_j|^2=\lambda_{r_j}.
\]
Elliptic regularity for the smooth Neumann problem gives $u_j\in C^\infty(\overline{\Om_{r_j}})$. By Lemma~\ref{lem:extension} and \eqref{eq:limsup}, the functions $U_j=\Phi_{r_j}u_j$ are bounded in $H^1(\Om)$. Along a subsequence,
\[
 U_j\rightharpoonup v\ \text{ weakly in }H^1(\Om),\qquad
 U_j\to v\ \text{ in }L^2(\Om).
\]
No mass is lost in the omitted region, since
\[
 \int_{\Om\setminus\Om_{r_j}}|U_j|^2
 \le2\norm{U_j-v}{L^2(\Om)}^2
       +2\int_{\Om\setminus\Om_{r_j}}|v|^2\longrightarrow0,
\]
and $\|U_j\|_{L^2(\Om)}^2=1+\int_{\Om\setminus\Om_{r_j}}|U_j|^2$. Hence $\int_\Om v^2=1$. Moreover, $\int_\Om U_j=\int_{\Om\setminus\Om_{r_j}}U_j\to0$ by the Cauchy--Schwarz inequality, because $|\Om\setminus\Om_{r_j}|\to0$, while $\int_\Om U_j\to\int_\Om v$ since $L^2(\Om)\subset L^1(\Om)$. So $\int_\Om v=0$.

\emph{Lower bound and identification.} Fix $s<1$. For large $j$ we have $U_j=u_j$ on $\Om_s$, so weak lower semicontinuity gives
\[
 \int_{\Om_s}|\nabla v|^2
 \le\liminf_j\int_{\Om_s}|\nabla U_j|^2
 \le\liminf_j\lambda_{r_j}.
\]
Letting $s\uparrow1$ and using \eqref{eq:variational},
\[
 \lambda\le\int_\Om|\nabla v|^2\le\liminf_j\lambda_{r_j}.
\]
Applying this to a sequence realizing $\liminf_{r\uparrow1}\lambda_r$, and combining with \eqref{eq:limsup}, proves \eqref{eq:eigenconv}. For an arbitrary sequence the same argument shows that $v$ is a normalized minimizer in \eqref{eq:variational}, hence
\[
 \int_\Om\nabla v\cdot\nabla h=\lambda\int_\Om vh
 \qquad(h\in H^1(\Om)).
\]

\emph{Local smooth convergence.} On every relatively compact subdomain, $-\Delta u_j=\lambda_{r_j}u_j$ holds for all large $j$, with bounded local $L^2$ norms and bounded eigenvalues. Interior elliptic estimates of every order, Sobolev embedding, and a diagonal subsequence give convergence in $C^\infty_{\mathrm{loc}}(\Om)$. The limit coincides with the $L^2$ limit $v$. If $\lambda$ is simple, then $v=\pm u$, and after changing signs of the $u_j$ the limit is $u$.
\end{proof}

\subsection{Passage to the hot-spots conclusion}\label{subsec:passage}
The following lemma isolates the step that converts a failure of \eqref{eq:HS} into an interior local maximum of an approximating function. It is used again in Sections~\ref{sec:HK} and~\ref{sec:withoutSimplicity}.

\begin{lemma}[Transfer of a compact superlevel set]\label{lem:superlevel}
Let $\Om\subset\R^2$ be a domain with $|\Om|<\infty$, let $u\in C(\Om)$ be real, and suppose that
\[
 \sup_\Om u>\lim_{\delta\downarrow0}\sup_{B_\delta(\Om)}u .
\]
Let $\Om_j\subset\Om$ be open sets such that every compact subset of $\Om$ lies in $\Om_j$ for all large $j$, and let $u_j\in C(\Om_j)$ converge to $u$ locally uniformly in $\Om$. Then, for all sufficiently large $j$, $u_j$ has a local maximum at some point of $\Om_j$.
\end{lemma}
\begin{proof}
Write $L=\lim_{\delta\downarrow0}\sup_{B_\delta(\Om)}u$ and $M=\sup_\Om u$. The limit exists in the extended real numbers because $B_\delta(\Om)$ decreases as $\delta$ decreases, and $L<M$ by assumption. Choose finite numbers $a,t$ with $L<a<t<M$, and $\delta>0$ with $u\le a$ on $B_\delta(\Om)$. Choose $x_0\in\Om$ with $u(x_0)>t$, and let $V$ be the connected component of $\{u>t\}$ containing $x_0$. Every point of $V$ satisfies $u>a$, hence lies outside $B_\delta(\Om)$, so $V\subset K_\delta$, where
\[
 K_\delta=\{x\in\Om:\dist(x,\partial\Om)\ge\delta\}.
\]
The set $K_\delta$ is compact. It is closed in $\R^2$. Indeed, a limit of points of $K_\delta$ is at distance at least $\delta$ from $\partial\Om$ and lies in $\overline\Om$. Since $\overline\Om\setminus\Om=\partial\Om$, it lies in $\Om$ and hence in $K_\delta$. It is also bounded. If it were unbounded, it would contain a sequence of points with mutual distances at least $\delta$, and the disks of radius $\delta/2$ about these points would be pairwise disjoint and contained in $\Om$, contradicting $|\Om|<\infty$. Hence $\overline V\subset K_\delta$ is a compact subset of $\Om$.

We claim that $u=t$ on $\partial V$. By continuity, $u\ge t$ on $\partial V$. If $u(y)>t$ at some $y\in\partial V$, a small ball around $y$ would lie in $\{u>t\}$ and meet $V$, hence lie in $V$. This contradicts $y\in\partial V$ because $V$ is open.

Put $\eta=u(x_0)-t>0$. For all large $j$ we have $\overline V\subset\Om_j$ and, by local uniform convergence, $\sup_{\overline V}|u_j-u|<\eta/3$. Consequently
\[
 u_j(x_0)>t+2\eta/3,\qquad
 \max_{\partial V}u_j<t+\eta/3.
\]
The maximum of $u_j$ on the compact set $\overline V$ is therefore attained at a point of the open set $V\subset\Om_j$, and that point is a local maximum of $u_j$ in $\Om_j$.
\end{proof}

\begin{proof}[Proof of Theorem~\ref{thm:main}]
Eigenfunctions are nonzero by definition, so we may take $u$ normalized. Proposition~\ref{prop:spectral} supplies $r_j\uparrow1$ and first positive eigenfunctions $u_j$ on $\Om_{r_j}$ with $u_j\to u$ locally uniformly in $\Om$. If the first identity in \eqref{eq:HS} failed, Lemma~\ref{lem:superlevel}, applied with $\Om_j=\Om_{r_j}$, would give an interior local maximum of a first positive Neumann eigenfunction on the bounded smooth simply connected domain $\Om_{r_j}$, contradicting Hypothesis~\ref{hyp:smooth}. Applying the same argument to $-u$ proves the identity for the infimum.

Now let $\Om$ be bounded and $u\in C(\overline\Om)$. By uniform continuity on $\overline\Om$, $\limsup_{\delta\downarrow0}\sup_{B_\delta(\Om)}u\le\max_{\partial\Om}u$: each point of $B_\delta(\Om)$ lies within $\delta$ of a boundary point. Conversely, a maximizing boundary point can be approached from inside $\Om$, which gives the opposite inequality. Together with $\sup_\Om u=\max_{\overline\Om}u$, this yields the first identity in \eqref{eq:boundary}. The second is analogous.
\end{proof}

\section{Heat-kernel bounds along the conformal exhaustion}\label{sec:HK}

Write $p_\Om(t,x,y)$ for the Neumann heat kernel of $\Om$, with heat equation $\partial_tu=\Delta u$. A \emph{polynomial heat-kernel bound} is a spatially uniform small-time estimate
\begin{equation}\label{eq:HK}
 0\le p_\Om(t,x,y)\le C_0t^{-N/2}\qquad(0<t\le1)
\end{equation}
for almost every $x,y\in\Om$ and some constants $N>2$ and $C_0<\infty$. Here $N$ is an exponent parameter, used in the Sobolev formulas below and in the stability theorems of \cite{BP}. The ambient dimension remains two. Bounded smooth and Lipschitz planar domains satisfy the sharp planar estimate $Ct^{-1}$, that is, \eqref{eq:HK} with $N=2$. Since $t^{-1}\le t^{-N/2}$ for $0<t\le1$, they satisfy \eqref{eq:HK} for every $N>2$ as well, and it is in this weakened form that the results of this section apply to them. Simple connectivity alone does not ensure any polynomial bound. Power cusps require exponents $N>2$ (Section~\ref{sec:examples}), and the exponential cusp of Section~\ref{sec:explicitCusp} has no polynomial bound at all.

\subsection{Uniform bounds along the conformal exhaustion}
The transfer result in this subsection is independent of Hypothesis~\ref{hyp:smooth} and of simplicity of the eigenvalue. It supplies the uniform estimates that allow the use of the stability theorems in \cite{BP}.

We use the following standard fact (Sobolev--ultracontractivity). Let $T_t$ be a symmetric sub-Markov semigroup associated with a nonnegative closed form $\mathcal E$ with form domain $\mathcal D(\mathcal E)$, and let $N>2$. If $\norm{T_t}{L^1\to L^\infty}\le Ct^{-N/2}$ for $0<t\le1$, then
\begin{equation}\label{eq:standardSob}
 \norm{w}{L^q}^2\le S\bigl(\mathcal E(w,w)+\norm{w}{L^2}^2\bigr),
 \qquad q=\frac{2N}{N-2},\qquad w\in\mathcal D(\mathcal E),
\end{equation}
with the constant $S$ depending only on $C$ and $N$. See \cite[Theorem~2.4.2]{Davies}. That theorem is stated for bounds valid for all $t>0$. To apply it, replace $T_t$ by $e^{-t}T_t$, whose form is $\mathcal E+\norm{\cdot}{L^2}^2$. For $t\ge1$ the semigroup property and the $L^1$ contraction property give $\norm{e^{-t}T_t}{L^1\to L^\infty}\le e^{-t}\norm{T_1}{L^1\to L^\infty}\le C'\,t^{-N/2}$. The converse implication needed below is proved explicitly by the Nash argument.

\begin{proposition}\label{prop:HK}
Let $\Om\subset\R^2$ be bounded and simply connected, and suppose that its Neumann heat kernel satisfies \eqref{eq:HK} for some constants $N>2$ and $C_0<\infty$. Let $F:\D\to\Om$ be a conformal bijection and $\Om_r=F(r\D)$. Then there exist $r_0<1$ and $C_*<\infty$ such that
\begin{equation}\label{eq:uniformHK}
 0\le p_{\Om_r}(t,x,y)\le C_*t^{-N/2}
 \qquad(r_0<r<1,\ 0<t\le1,\ x,y\in\Om_r).
\end{equation}
The exponent is unchanged, and $C_*$ is independent of $r$.
\end{proposition}
\begin{proof}
By \eqref{eq:standardSob}, the bound \eqref{eq:HK} gives
\begin{equation}\label{eq:targetSob}
 \norm{w}{L^q(\Om)}^2\le S\left(
 \int_\Om|\nabla w|^2+\int_\Om|w|^2\right),
 \qquad q=\frac{2N}{N-2},\quad w\in H^1(\Om).
\end{equation}
We transfer this inequality to $\Om_r$ without comparing heat kernels under inclusion.

\emph{Step 1: absorption.} Take $v\in C^\infty(\overline{\Om_r})$ and let $V=\Phi_rv$ be the bounded extension of Lemma~\ref{lem:extension}. Its construction and \eqref{eq:energy} do not use compactness. Put
\[
 m_r=|\Om\setminus\Om_r|,\quad
 X=\norm{V}{L^q(\Om)}^2,\quad
 A_r=\norm{v}{L^2(\Om_r)}^2,\quad
 B_r=\norm{\nabla v}{L^2(\Om_r)}^2.
\]
All these quantities are finite since $V$ is bounded and $|\Om|<\infty$. H\"older's inequality with $1-2/q=2/N$ gives
\begin{equation}\label{eq:holder}
 \int_{\Om\setminus\Om_r}|V|^2\le m_r^{2/N}X.
\end{equation}
Combining \eqref{eq:targetSob}, \eqref{eq:energy} and \eqref{eq:holder}, we obtain
\[
 X\le S\bigl(2B_r+A_r+m_r^{2/N}X\bigr).
\]
Since $m_r\downarrow0$ as $r\uparrow1$, choose $r_0$ such that $Sm_r^{2/N}\le\tfrac12$ for $r>r_0$. Absorbing the last term,
\begin{equation}\label{eq:uniformSob}
 \norm{v}{L^q(\Om_r)}^2\le X\le2S(2B_r+A_r)
 \le4S\left(\norm{\nabla v}{L^2(\Om_r)}^2+
                    \norm{v}{L^2(\Om_r)}^2\right).
\end{equation}
Because $\Om_r$ has smooth boundary, $C^\infty(\overline{\Om_r})$ is dense in $H^1(\Om_r)$. Applying \eqref{eq:uniformSob} to differences of smooth approximants shows that they are Cauchy in $L^q$, and their $L^q$ limit agrees with their $L^2$ limit. Hence \eqref{eq:uniformSob} holds for all $v\in H^1(\Om_r)$ with the same constant $K=4S$, independent of $r$.

\emph{Step 2: from Sobolev to the heat-kernel bound.} Interpolation between $L^1$ and $L^q$ gives
\[
 \norm{v}{L^2}\le\norm{v}{L^1}^{2/(N+2)}
                       \norm{v}{L^q}^{N/(N+2)},
\]
so \eqref{eq:uniformSob} implies the Nash inequality
\begin{equation}\label{eq:Nash}
 \norm{v}{L^2}^{2+4/N}\le
 K\left(\norm{\nabla v}{L^2}^2+\norm{v}{L^2}^2\right)
 \norm{v}{L^1}^{4/N}.
\end{equation}
All norms below are taken on $\Om_r$. Let $S_t=e^{-t(A_{\Om_r}+I)}$, an $L^1$ contraction, and let $f\in L^1\cap L^2$ with $f\ne0$. For $y(t)=\norm{S_tf}{L^2}^2$, the spectral theorem and \eqref{eq:Nash} give, for $t>0$,
\[
 y'(t)=-2\left(\norm{\nabla S_tf}{L^2}^2+
                       \norm{S_tf}{L^2}^2\right)
 \le-\frac{2}{K\norm{f}{L^1}^{4/N}}\,y(t)^{1+2/N}.
\]
If $y$ vanishes at some $t_0$, it vanishes for all $t\ge t_0$ by the semigroup property and the bound below is trivial. Otherwise integration yields
\[
 y(t)\le\left(\frac{NK}{4t}\right)^{N/2}\norm{f}{L^1}^2.
\]
By density, duality, and $S_t=S_{t/2}S_{t/2}$,
\[
 \norm{S_t}{L^1\to L^\infty}
 \le\left(\frac{NK}{2t}\right)^{N/2}.
\]
Since $T_t^{\Om_r}=e^tS_t$, for $0<t\le1$ we obtain
\[
 \norm{T_t^{\Om_r}}{L^1\to L^\infty}
 \le e\left(\frac{NK}{2}\right)^{N/2}t^{-N/2}.
\]
This is the kernel bound almost everywhere with $C_*=e(NK/2)^{N/2}$. Interior continuity of the heat kernel of the smooth domain $\Om_r$ extends it to every $x,y\in\Om_r$.
\end{proof}

\begin{remark}[Smooth targets and the endpoint]\label{rem:endpoint}
If the target domain satisfies $p_\Om(t,x,y)\le Ct^{-1}$, Proposition~\ref{prop:HK} gives, for every fixed $\varepsilon>0$, constants $r_\varepsilon<1$ and $C_\varepsilon<\infty$ such that
\[
 p_{\Om_r}(t,x,y)\le C_\varepsilon t^{-1-\varepsilon},
 \qquad r_\varepsilon<r<1,\quad 0<t\le1.
\]
The constants may depend on $\varepsilon$, and no passage to the endpoint $\varepsilon=0$ is justified. Although each smooth $\Om_r$ separately has a bound $C_rt^{-1}$, this does not show that $C_r$ remains bounded as $r\uparrow1$. The uniform weaker bound suffices for the application of \cite{BP} below. The transfer result concerns the conformal exhaustion. It neither proves nor disproves uniform preservation of the endpoint.
\end{remark}

\subsection{A second proof via spectral stability}\label{subsec:stability}

\begin{remark}[A heat-kernel bound implies compact embedding]\label{rem:adds}
Write $\|\cdot\|_{\mathrm{HS}}$ for the Hilbert--Schmidt norm on $L^2(\Om)$. Since $1\in H^1(\Om)$ has zero form energy, $T_t^\Om 1=1$. For each fixed $0<t\le1$, conservativity and \eqref{eq:HK} give
\[
 \norm{T_t^\Om}{\mathrm{HS}}^2
 =\int_\Om\int_\Om p_\Om(t,x,y)^2\,dy\,dx
 \le C_0t^{-N/2}|\Om|<\infty,
\]
and for $t>1$ the factorization $T_t^\Om=T_1^\Om T_{t-1}^\Om$ shows that $T_t^\Om$ is Hilbert--Schmidt as well. Thus the semigroup is compact for every positive time. Compactness of the resolvent follows from
\[
 (A_\Om+I)^{-1}=\int_0^\infty e^{-t}T_t^\Om\,dt:
\]
the integral over $[\varepsilon,M]$ is compact (factor out $T_{\varepsilon/2}^\Om$), while the operator norms of the omitted integrals tend to zero as $\varepsilon\downarrow0$ and $M\uparrow\infty$. Compactness of the resolvent is equivalent to compactness of the inclusion of the form domain $H^1(\Om)$ into $L^2(\Om)$. Hence \eqref{eq:HK} implies \eqref{eq:CI}.
\end{remark}

\begin{corollary}\label{cor:HKhot}
Assume Hypothesis~\ref{hyp:smooth}. Let $\Om$ be bounded and simply connected and satisfy \eqref{eq:HK} for some $N>2$. If its first positive Neumann eigenvalue is simple, then \eqref{eq:HS} holds for every real eigenfunction belonging to that eigenvalue, and \eqref{eq:boundary} holds for those eigenfunctions that extend continuously to $\overline\Om$.
\end{corollary}

By Remark~\ref{rem:adds}, Corollary~\ref{cor:HKhot} is a special case of Theorem~\ref{thm:main}. It includes every target satisfying $Ct^{-1}$ and having a simple first positive eigenvalue, by the choice of any fixed $N>2$. We give a second proof, which uses Proposition~\ref{prop:HK} together with the stability theory of \cite{BP} in place of Proposition~\ref{prop:spectral}. Its additional content is that the uniform bounds required by \cite{BP} follow from a bound on the target alone, for conformal exhaustions, by the extension and absorption argument rather than by monotonicity of Neumann kernels. Pang's sequel \cite{PangII} treats multiplicity and verifies uniform bounds for the usual polygonal approximation of the Koch snowflake.

\begin{proof}[Second proof of Corollary~\ref{cor:HKhot}, via \cite{BP}]
Proposition~\ref{prop:HK} supplies the bound $p_{\Om_r}(t,x,y)\le C_*t^{-N/2}$ for $r_0<r<1$ and $0<t\le1$, and \eqref{eq:HK} supplies the bound on $\Om$. Put $c_0=\max\{1,C_0,C_*\}$, so that one constant serves $\Om$ and all approximants with $r>r_0$, as \cite[Hypothesis 1.1]{BP} requires. We use the interior-continuous versions of these kernels, so the bounds hold at every interior pair of points, as in \cite[Proposition 3]{BD}. The geometric condition of \cite{BP} is that the boundary of the approximant labeled by $\delta$ lies in the $\delta$-layer of $\partial\Om$. Put
\[
 d_r=\sup_{x\in\partial\Om_r}\dist(x,\partial\Om).
\]
Then $d_r\to0$: for $\varepsilon>0$ the set $K_\varepsilon=\{x\in\Om:\dist(x,\partial\Om)\ge\varepsilon\}$ is compact (see the proof of Lemma~\ref{lem:superlevel}), so $K_\varepsilon\subset\Om_r$ for large $r$, whence $\partial\Om_r\cap K_\varepsilon=\varnothing$ and $d_r\le\varepsilon$. Choose $r_j\uparrow1$ with $r_j>r_0$ and $d_{r_j}\le1/j$, and label the approximants by $\Om_\delta=\Om_{r_j}$ for $1/j\le\delta<1/(j-1)$, $j\ge2$.

Theorems 1.3, 1.4 and 1.5 of \cite{BP} then give convergence of the first positive eigenvalues (indexed $\mu_2$ there), eventual simplicity of the first positive eigenvalue on the approximants, and uniform convergence on compact subsets of $\Om$ of the $L^2$-normalized first eigenfunctions $u_j$ of $\Om_{r_j}$ to the prescribed normalized eigenfunction $u$, once the sign of each $u_j$ is fixed by $\int_{\Om_{r_j}}u_ju>0$, as in the proof of \cite[Theorem 1.5]{BP}. Two remarks on the scope of these theorems are in order. First, \cite[Theorem 1.5]{BP} states uniform convergence on subdomains $\Om_0$ of $\Om$, and its proof uses that $\overline{\Om_0}$ is a compact subset of $\Om$. This is the form used here. Second, \cite[Theorem 1.3]{BP} rests on its Lemma 2.2, a semigroup-convergence statement taken from Burdzy--Chen \cite{BC}; Theorem 2 of \cite{BC} is proved for bounded targets with continuous boundary and increasing approximants with continuous boundaries, and its extension to targets without boundary regularity is described in \cite[Section 3]{BC}. For a target with no boundary regularity this second proof therefore depends on that extension, whereas the first proof, through Theorem~\ref{thm:main}, does not.

With the locally uniform convergence in hand, Lemma~\ref{lem:superlevel} applies exactly as in the proof of Theorem~\ref{thm:main}.
\end{proof}

\begin{remark}[Relation to reflecting Brownian motion]
Burdzy--Chen \cite{BC} prove weak convergence of reflecting Brownian motions along increasing-domain approximations, with the definitions and qualifications specified there. This supplies a semigroup-convergence input in the Ba\~{n}uelos--Pang theory. Weak convergence alone does not provide the uniform heat-kernel bounds used here, nor does it by itself establish compact Sobolev embedding. Section~\ref{sec:main} instead uses energy estimates and compactness directly.
\end{remark}

\section{Domains with polynomial heat-kernel bounds}\label{sec:examples}
This section lists classes of domains satisfying \eqref{eq:HK}. The bounds concern $0<t\le1$ and allow constants depending on the domain. By Proposition~\ref{prop:HK}, a target bound with exponent $N>2$ yields uniform bounds on its conformal exhaustion, and by Remark~\ref{rem:adds} it implies compact embedding. Simplicity of the first positive Neumann eigenvalue is a separate matter, treated in Section~\ref{sec:simplicity}. Further classes (quasihyperbolic H\"older domains, finite decompositions, a slit disk) are collected in Appendix~\ref{app:classes}.

\subsection{The Sobolev criterion}
For a planar domain $\Om$ and $2<q<\infty$, the continuous embedding estimate
\begin{equation}\label{eq:examplesSob}
 \norm{v}{L^q(\Om)}^2\le C\left(
 \norm{\nabla v}{L^2(\Om)}^2+\norm{v}{L^2(\Om)}^2\right)
 \qquad(v\in H^1(\Om))
\end{equation}
is equivalent to a polynomial Neumann heat-kernel bound with exponent $N=2q/(q-2)>2$. See \cite[Proposition 3]{BD}, where this exponent is denoted $M$ and $N$ is the dimension. The condition is a continuous embedding into $L^q$, not merely compactness into $L^2$. It is used in Section~\ref{sec:explicitCusp} and in Remark~\ref{rem:hornNoPolynomial} to show that certain domains have no polynomial bound.

\subsection{Lipschitz domains, quasidisks, and H\"older graph domains}
Bounded Lipschitz domains, including nonconvex polygons, satisfy $p_\Om(t,x,y)\le Ct^{-1}$ for $0<t\le1$. See \cite[Remark 1.2]{BP}.

A further class consists of bounded quasidisks, that is, images of the disk under quasiconformal homeomorphisms of the plane. They include the Koch snowflake domain. Quasidisks are uniform domains, and the Neumann heat kernel of a bounded uniform planar domain satisfies $Ct^{-1}$ by the two-sided Gaussian estimates of Gyrya--Saloff-Coste \cite{GSC} and Murugan \cite{Murugan}, although the boundary need not be rectifiable or locally representable as a graph. Here the small-scale volume bound $|B(x,r)\cap\Om|\ge c r^2$ on a uniform planar domain converts the volume denominator in the Gaussian estimate into $Ct^{-1}$. The snowflake and its polygonal approximations are also treated in \cite[Theorem 1.3]{PangII}. For the weaker polynomial bound needed here it suffices that uniform domains have the Sobolev extension property \cite{Jones}, which gives \eqref{eq:examplesSob} for every finite $q>2$.

Suppose $\Om$ is a bounded planar domain whose boundary is locally, after a rotation and translation, the graph of a H\"older continuous function of exponent $0<\gamma<1$, with $\Om$ lying locally on one side of the graph (the class $\operatorname{Lip}\gamma$ of \cite{BD}). Then \cite[Proposition 3]{BD} gives
\begin{equation}\label{eq:HolderExponents}
 q=\frac{2(1+\gamma)}{1-\gamma},\qquad
 N=1+\frac1\gamma,
 \qquad p_\Om(t,x,y)\le Ct^{-(1+1/\gamma)/2}.
\end{equation}

\subsection{Power cusps}\label{subsec:powerCusps}
For $\alpha>1$, consider the bounded simply connected domain
\[
 \Om_\alpha=\{(x,y):0<x<1,\ |y|<x^\alpha\}.
\]
Near the tip its boundary is the graph $x=|y|^{1/\alpha}$, with the domain on the side $x>|y|^{1/\alpha}$. The remaining boundary is Lipschitz. Figure~\ref{fig:profiles} compares the profile curves $y=x^\alpha$ for several $\alpha$ with the exponential profile treated in Section~\ref{sec:simplicity}. Applying \eqref{eq:HolderExponents} with $\gamma=1/\alpha$ gives
\begin{equation}\label{eq:powerCuspHK}
 p_{\Om_\alpha}(t,z,w)\le C_\alpha t^{-(\alpha+1)/2}.
\end{equation}
This is the planar specialization of \cite[Example 6]{BD}, and the exponent is optimal for this family. The domain is not Lipschitz at the tip. These cusps also have positive quasihyperbolic order. See Appendix~\ref{app:classes}.

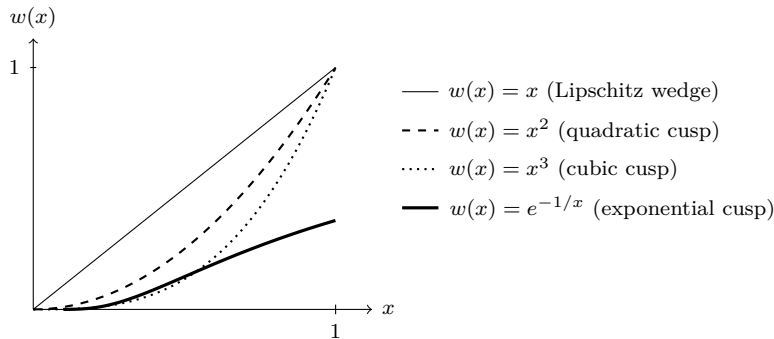
\begin{figure}[ht]
\centering
\begin{tikzpicture}[x=4cm,y=3.2cm]
\draw[->] (0,0) -- (1.12,0) node[right,font=\scriptsize] {$x$};
\draw[->] (0,0) -- (0,1.12) node[above,font=\scriptsize] {$w(x)$};
\draw (1,0.02) -- (1,-0.02) node[below,font=\scriptsize] {$1$};
\draw (0.01,1) -- (-0.01,1) node[left,font=\scriptsize] {$1$};
\draw[thin] (0,0) -- (1,1);
\draw[thick,dashed] plot[domain=0:1,samples=40] ({\x},{\x*\x});
\draw[thick,dotted] plot[domain=0:1,samples=40] ({\x},{\x*\x*\x});
\draw[very thick] plot[domain=0.1:1,samples=60] ({\x},{exp(-1/\x)});
\draw[thin] (1.22,0.9) -- (1.34,0.9) node[right,font=\scriptsize] {$w(x)=x$ (Lipschitz wedge)};
\draw[thick,dashed] (1.22,0.74) -- (1.34,0.74) node[right,font=\scriptsize] {$w(x)=x^{2}$ (quadratic cusp)};
\draw[thick,dotted] (1.22,0.58) -- (1.34,0.58) node[right,font=\scriptsize] {$w(x)=x^{3}$ (cubic cusp)};
\draw[very thick] (1.22,0.42) -- (1.34,0.42) node[right,font=\scriptsize] {$w(x)=e^{-1/x}$ (exponential cusp)};
\end{tikzpicture}
\caption{Profile curves $y=w(x)$, $0<x<1$, bounding the domains $\{0<x<1,\ |y|<w(x)\}$. The power profiles $w(x)=x^\alpha$ give the cusps $\Om_\alpha$ ($\alpha=1$ is a Lipschitz wedge). The profile $e^{-1/x}$ is flatter than every power at the tip and leads to the exponential cusp of Theorem~\ref{thm:explicitExpCusp}.}
\label{fig:profiles}
\end{figure}

\section{Cross-cut domains and simplicity of the first positive eigenvalue}\label{sec:simplicity}
All geometric, compactness, and spectral estimates in this section and Section~\ref{sec:unboundedFiniteArea} are independent of Hypothesis~\ref{hyp:smooth}. That hypothesis is used only for the hot-spots conclusions.

When the Sobolev embedding is compact, simplicity of $\lambda_1$ is precisely the inequality $\lambda_1<\lambda_2$. Heat-kernel bounds alone do not imply it. The disk and the square are counterexamples. In this section we verify it for a family of nonsmooth simply connected domains, all of which lie between two graphs. For such \emph{cross-cut domains} the Neumann eigenvalues can be compared with those of an explicit weighted one-dimensional problem. This comparison replaces the compactness-based thin-domain arguments and gives explicit thresholds. Figure~\ref{fig:examples} shows the three families treated below, and Table~\ref{tab:examples} in Section~\ref{subsec:summary} summarizes the results.

\begin{figure}[ht]
\centering
\begin{tikzpicture}
\begin{scope}[shift={(0cm,-1.1cm)},x=2cm,y=2cm]
\filldraw[fill=gray!20,draw=black,thick] (0,0.1) -- (1,0) -- (2,0.1) -- (2,1.1) -- (1,1) -- (0,1.1) -- cycle;
\fill (1,1) circle (1.2pt);
\node[above,font=\scriptsize] at (1,1.08) {reentrant corner};
\end{scope}
\begin{scope}[shift={(5.2cm,0cm)},x=3.4cm,y=3.4cm]
\filldraw[fill=gray!20,draw=black,thick] (0,0) -- plot[domain=0:1,samples=40] ({\x},{0.4*\x*\x}) -- (1,-0.4) -- plot[domain=1:0,samples=40] ({\x},{-0.4*\x*\x}) -- cycle;
\draw[dashed,gray] (-0.05,0) -- (1.05,0);
\node[left,font=\scriptsize] at (0,0) {tip};
\end{scope}
\begin{scope}[shift={(10.2cm,0cm)},x=3.4cm,y=3.4cm]
\filldraw[fill=gray!20,draw=black,thick] (0,0) -- plot[domain=0.1:1,samples=60] ({\x},{exp(-1/\x)}) -- (1,-0.3679) -- plot[domain=1:0.1,samples=60] ({\x},{-exp(-1/\x)}) -- cycle;
\draw[dashed,gray] (-0.05,0) -- (1.05,0);
\node[left,font=\scriptsize] at (0,0) {tip};
\end{scope}
\node[below,font=\small,align=center] at (2cm,-1.5cm) {(a) $T_h$,\\ $h(x)=|x-1|/10$};
\node[below,font=\small,align=center] at (6.9cm,-1.5cm) {(b) $Q_\varepsilon$,\\ $|y|<\varepsilon x^2$};
\node[below,font=\small,align=center] at (11.9cm,-1.5cm) {(c) $E_\varepsilon$,\\ $|y|<\varepsilon e^{-1/x}$};
\end{tikzpicture}
\caption{The explicit domains of Section~\ref{sec:simplicity}. (a) A Lipschitz tube $T_h$ (Proposition~\ref{prop:tubes}) with a reentrant corner at $(1,1)$. (b) The quadratic cusp $Q_\varepsilon$ (Theorem~\ref{thm:quadCusp}) and (c) the exponential cusp $E_\varepsilon$ (Theorem~\ref{thm:explicitExpCusp}), drawn with $\varepsilon=0.4$ and $\varepsilon=1$ so that the shapes are visible. The theorems concern the thinner domains $\varepsilon\le1/5$ and $\varepsilon\le1/10$. Dashed lines mark $y=0$.}
\label{fig:examples}
\end{figure}
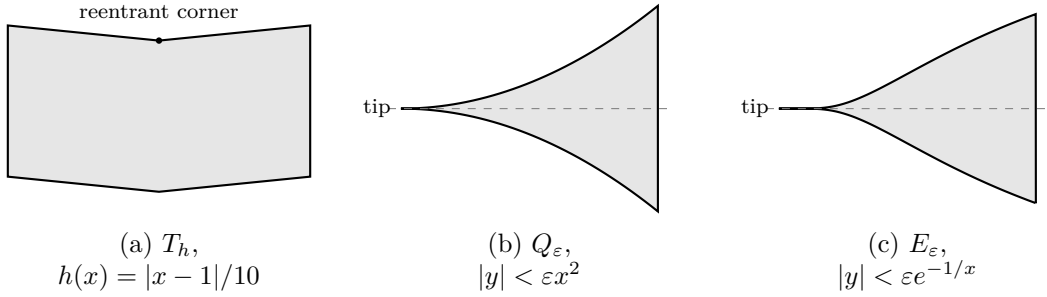

\subsection{Cross-cut domains}\label{subsec:crosscutdef}
\begin{definition}\label{def:crosscut}
Let $L>0$ and let $a<b$ be locally Lipschitz functions on $(0,L)$. The \emph{cross-cut domain} is
\[
 G_{a,b}=\{(x,y):0<x<L,\ a(x)<y<b(x)\},
\]
and $\ell=b-a$ is its \emph{cross-cut length}. The vertical segment over $x$ is the cross-cut at $x$ (Figure~\ref{fig:crosscuts}). We assume
\[
 \ell_*=\sup_{0<x<L}\ell(x)<\infty,\qquad
 \kappa_*^2=\operatorname*{ess\,sup}_{0<x<L}\frac{a'(x)^2+a'(x)b'(x)+b'(x)^2}{3}<\infty .
\]
\end{definition}
The map $(x,t)\mapsto(x,a(x)+\ell(x)t)$ is a homeomorphism of $(0,L)\times(0,1)$ onto $G_{a,b}$, so $G_{a,b}$ is simply connected. The bound on $\kappa_*$ implies global bounds on $a'$ and $b'$, since $a'^2+a'b'+b'^2\ge(a'^2+b'^2)/2$. Thus both functions extend to Lipschitz functions on $[0,L]$, and $G_{a,b}$ is bounded. Note that $\kappa_*\le\max\{\|a'\|_\infty,\|b'\|_\infty\}$. In the symmetric case $a=-h$, $b=h$ one has $\ell=2h$ and $\kappa_*=\|h'\|_\infty/\sqrt3$. For a sheared strip $a=h$, $b=1+h$ one has $\ell=1$ and $\kappa_*=\|h'\|_\infty$.

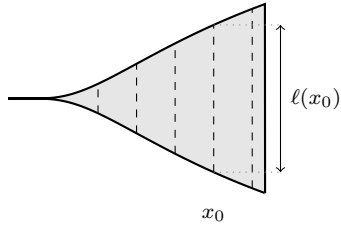
\begin{figure}[ht]
\centering
\begin{tikzpicture}[x=3.4cm,y=3.4cm]
\filldraw[fill=gray!20,draw=black,thick] (0,0) -- plot[domain=0.1:1,samples=60] ({\x},{exp(-1/\x)}) -- (1,-0.3679) -- plot[domain=1:0.1,samples=60] ({\x},{-exp(-1/\x)}) -- cycle;
\foreach \a in {0.35,0.5,0.65,0.8,0.95}{\draw[dashed] (\a,{-exp(-1/\a)}) -- (\a,{exp(-1/\a)});}
\draw[thin,gray,dotted] (.8,{-exp(-1/.8)}) -- (1.06,{-exp(-1/.8)});
\draw[thin,gray,dotted] (.8,{exp(-1/.8)}) -- (1.06,{exp(-1/.8)});
\draw[<->,thin] (1.06,{-exp(-1/.8)}) -- (1.06,{exp(-1/.8)}) node[midway,right,font=\scriptsize] {$\ell(x_0)$};
\node[below,font=\scriptsize] at (.8,-.39) {$x_0$};
\end{tikzpicture}
\caption{Cross-cuts (dashed) of a cross-cut domain, here the profile $h(x)=e^{-1/x}$ with $a=-h$, $b=h$. The marked cross-cut is at $x_0=0.8$. The cross-cut over $x$ is an interval of length $\ell(x)$. Its Neumann spectrum is $(n\pi/\ell(x))^2$, $n\ge0$.}
\label{fig:crosscuts}
\end{figure}

The \emph{longitudinal problem} is the following weighted one-dimensional problem. Let
\[
 \mathcal V_\ell=\Bigl\{f\in H^1_{\mathrm{loc}}(0,L):\ \int_0^L\ell\,(|f|^2+|f'|^2)\,dx<\infty\Bigr\},
\]
and let $0=\mu_0\le\mu_1\le\mu_2\le\cdots$ be the min--max values of
\[
 \frac{\int_0^L\ell\,|f'|^2\,dx}{\int_0^L\ell\,f^2\,dx},\qquad f\in\mathcal V_\ell ,
\]
defined over $(k+1)$-dimensional subspaces of $\mathcal V_\ell$ exactly as the $\lambda_k$ in Section~\ref{sec:main}. No compactness is assumed in this definition. A function $u(x,y)=f(x)$ belongs to $H^1(G_{a,b})$ if and only if $f\in\mathcal V_\ell$, and then $\int_{G_{a,b}}u^2=\int_0^L\ell f^2\,dx$ and $\int_{G_{a,b}}|\nabla u|^2=\int_0^L\ell|f'|^2\,dx$. If $H^1(G_{a,b})\hookrightarrow L^2(G_{a,b})$ is compact, then so is $\mathcal V_\ell\hookrightarrow L^2(\ell\,dx)$, since $y$-independent functions form a closed subspace. In that case the $\mu_k$ are eigenvalues of a self-adjoint operator.

\begin{remark}[Cross-cut techniques in the literature]\label{rem:crosscutlit}
The cross-cut approach below is motivated by the classical use of cross-cuts in geometric function theory and probabilistic potential theory. Decomposing a planar domain along arcs joining boundary points plays an important role in estimates for conformal maps and harmonic measure. See \cite{Ahlfors,GM}. Related geometric ideas enter estimates for Brownian exit times and the Dirichlet fundamental frequency. See for example \cite{BCDuke,BCAIF,BVan,Hayman}. For the Neumann problem considered here, the comparison uses the cross-cut decomposition directly, without a conformal map. On each vertical cross-cut, the transverse Neumann spectrum is explicit, while the slopes of the bounding graphs control the interaction between the longitudinal and transverse directions. This mechanism is closely related to dimension reduction in thin domains \cite{HaleRaugel,Jimbo} and the analysis of Neumann Laplacians on cuspidal domains \cite{HSS}.

The lip domains of Atar and Burdzy \cite{AB} form a subclass of the cross-cut domains considered here. Their theorem establishes simplicity of the first positive Neumann eigenvalue for every nonsquare lip domain, whereas the sufficient condition in Corollary~\ref{cor:crosscutSimple} does not recover that full conclusion. Conversely, their theorem requires the domain itself to be Lipschitz and therefore does not cover the quadratic and exponential cusp domains treated below. For these non-Lipschitz examples, our comparison gives explicit spectral-gap bounds and sufficient parameter ranges for simplicity. See Theorems~\ref{thm:quadCusp} and~\ref{thm:explicitExpCusp}. The results therefore overlap on lip domains satisfying Corollary~\ref{cor:crosscutSimple}, where simplicity is already known from \cite{AB}. For an earlier probabilistic approach to the hot spots conjecture, see \cite{BB}.
\end{remark}

\subsection{The cross-cut comparison principle}
\begin{proposition}[Cross-cut comparison]\label{prop:crosscut}
Let $G=G_{a,b}$ be a cross-cut domain as in Definition~\ref{def:crosscut}, let $\lambda_k(G)$ be its Neumann min--max values as in Section~\ref{sec:main}, and put $c=(\ell_*/\pi)^2$. Then for every $k\ge0$,
\begin{equation}\label{eq:crosscut}
 \frac{\mu_k}{1+\kappa_*^2+c\,\mu_k}\ \le\ \lambda_k(G)\ \le\ \mu_k .
\end{equation}
No compactness is assumed.
\end{proposition}
\begin{proof}
\emph{Upper bound.} The functions $u(x,y)=f(x)$, $f\in\mathcal V_\ell$, lie in $H^1(G)$ with the same Rayleigh quotient as $f$. Min--max over this subspace gives $\lambda_k(G)\le\mu_k$.

\emph{Lower bound.} Fix $k$ and $\eta>0$, and put $\Lambda=\lambda_k(G)+\eta$. By the min--max characterization there is a $(k+1)$-dimensional subspace $W\subset H^1(G)$ of real functions with $\int_G|\nabla u|^2\le\Lambda\int_Gu^2$ for all $u\in W$. For $u\in W$ write $A=\|u_x\|^2_{L^2(G)}$ and $B=\|u_y\|^2_{L^2(G)}$, and define the cross-cut average and fluctuation
\[
 \bar u(x)=\frac1{\ell(x)}\int_{a(x)}^{b(x)}u(x,y)\,dy,\qquad \tilde u=u-\bar u .
\]

\emph{(1) Fluctuation.} For almost every $x$, $u(x,\cdot)\in H^1(a(x),b(x))$ and $\tilde u(x,\cdot)$ has mean zero. The first nonzero Neumann eigenvalue of an interval of length $\ell$ is $(\pi/\ell)^2$, so the Poincar\'e--Wirtinger inequality on each cross-cut gives
\[
 \int_{a}^{b}|\tilde u|^2\,dy\le\Bigl(\frac{\ell}{\pi}\Bigr)^2\int_a^b|u_y|^2\,dy,
 \qquad\text{hence}\qquad
 \|\tilde u\|^2_{L^2(G)}\le c\,B .
\]
Since $\bar u$ and $\tilde u$ are orthogonal on each cross-cut,
\begin{equation}\label{eq:crosscutOrth}
 \|u\|^2_{L^2(G)}=\int_0^L\ell\,\bar u^2\,dx+\|\tilde u\|^2_{L^2(G)} .
\end{equation}

\emph{(2) Averaged energy.} Fix a compact interval $J\Subset(0,L)$. The map $(x,t)\mapsto(x,a(x)+\ell(x)t)$ is bi-Lipschitz from $J\times(0,1)$ onto the corresponding portion of $G$, because $\ell$ is bounded away from zero on $J$. Thus, for $u\in H^1(G)$, its pullback $v$ belongs to $H^1(J\times(0,1))$. Sobolev slicing and the weak chain rule give $\bar u(x)=\int_0^1v(x,t)\,dt$ and, almost everywhere,
\[
 v_x=u_x+\bigl(a'+\ell't\bigr)\,u_y,\qquad\text{evaluated at }(x,a+\ell t).
\]
Taking the weak derivative of this average,
\[
 \bar u'(x)=\int_0^1u_x\,dt+\int_0^1\bigl(a'+\ell't\bigr)\,u_y\,dt .
\]
By the Cauchy--Schwarz inequality, $\bigl|\int_0^1u_x\,dt\bigr|\le\ell^{-1/2}\|u_x(x,\cdot)\|_{L^2(a,b)}$. Since $a'+\ell't$ is affine in $t$ with end values $a'$ and $b'$, we have $\int_0^1(a'+\ell't)^2dt=(a'^2+a'b'+b'^2)/3\le\kappa_*^2$, and therefore
\[
 \Bigl|\int_0^1\bigl(a'+\ell't\bigr)\,u_y\,dt\Bigr|\le\kappa_*\,\ell^{-1/2}\|u_y(x,\cdot)\|_{L^2(a,b)} .
\]
Consequently
\[
 \sqrt{\ell}\,|\bar u'|\le\|u_x(x,\cdot)\|_{L^2(a,b)}+\kappa_*\|u_y(x,\cdot)\|_{L^2(a,b)},
\]
and Minkowski's inequality in $L^2(dx)$ together with the Cauchy--Schwarz inequality give
\begin{equation}\label{eq:crosscutEnergy}
 \int_0^L\ell\,|\bar u'|^2\,dx\le\bigl(\sqrt A+\kappa_*\sqrt B\bigr)^2\le(1+\kappa_*^2)(A+B).
\end{equation}
Since $J$ was arbitrary, the derivative identity holds locally throughout $(0,L)$. Integration and monotone exhaustion by such intervals give \eqref{eq:crosscutEnergy} globally. Jensen's inequality also gives $\int_0^L\ell\bar u^2\le\|u\|_{L^2(G)}^2$. Hence $\bar u\in\mathcal V_\ell$.

\emph{(3) Min--max.} Suppose $c\Lambda<1$. For $u\in W$, \eqref{eq:crosscutOrth} and $B\le\Lambda\|u\|^2$ give
\[
 \int_0^L\ell\,\bar u^2\,dx\ge(1-c\Lambda)\|u\|^2_{L^2(G)}>0 \qquad(u\ne0),
\]
so $u\mapsto\bar u$ is injective on $W$, and its image is a $(k+1)$-dimensional subspace of $\mathcal V_\ell$. By \eqref{eq:crosscutEnergy} and $A+B\le\Lambda\|u\|^2$, every element of the image has Rayleigh quotient at most
\[
 \frac{(1+\kappa_*^2)\Lambda}{1-c\Lambda}.
\]
By min--max for the longitudinal problem, $\mu_k(1-c\Lambda)\le(1+\kappa_*^2)\Lambda$, that is, $\Lambda\ge\mu_k/(1+\kappa_*^2+c\mu_k)$. If $c\Lambda\ge1$, the same lower bound holds trivially because $\mu_k/(1+\kappa_*^2+c\mu_k)<1/c$. Letting $\eta\downarrow0$ proves \eqref{eq:crosscut}.
\end{proof}

\begin{corollary}[Simplicity criterion]\label{cor:crosscutSimple}
Assume that $H^1(G_{a,b})\hookrightarrow L^2(G_{a,b})$ is compact. If $(1+\kappa_*^2)\,\mu_1+c\,\mu_1\mu_2<\mu_2$, then $\lambda_1(G_{a,b})<\lambda_2(G_{a,b})$.
\end{corollary}
\begin{proof}
By \eqref{eq:crosscut}, $\lambda_1\le\mu_1<\mu_2/(1+\kappa_*^2+c\mu_2)\le\lambda_2$, where the middle inequality is the hypothesis.
\end{proof}

Note that $\mu\mapsto\mu/(1+\kappa_*^2+c\mu)$ is increasing, so in applications one may insert a lower bound for $\mu_2$ and an upper bound for $\mu_1$.

\subsection{Thin profiles}
Let $w:(0,1)\to(0,\infty)$ be locally Lipschitz with $W_0=\sup w<\infty$ and $W_1=\operatorname{ess\,sup}|w'|<\infty$, and let
\[
 G^w_\varepsilon=\{(x,y):0<x<1,\ |y|<\varepsilon\,w(x)\},\qquad\varepsilon>0 .
\]
This is the cross-cut domain with $a=-\varepsilon w$, $b=\varepsilon w$, so $\ell=2\varepsilon w$, $\ell_*=2\varepsilon W_0$ and $\kappa_*=\varepsilon W_1/\sqrt3$. The longitudinal problem has the weight $2\varepsilon w$, and the constant $2\varepsilon$ cancels in the Rayleigh quotient. Hence its min--max values $\mu_k=\mu_k(w)$ are independent of $\varepsilon$ and are those of $\int_0^1w|f'|^2$ relative to $\int_0^1wf^2$.

\begin{corollary}[Thin profiles]\label{cor:thin}
With the notation above, for all $\varepsilon>0$ and $k\ge0$,
\begin{equation}\label{eq:thin}
 \frac{\mu_k}{1+\varepsilon^2\bigl(W_1^2/3+(4W_0^2/\pi^2)\,\mu_k\bigr)}\ \le\ \lambda_k(G^w_\varepsilon)\ \le\ \mu_k .
\end{equation}
In particular $\lambda_k(G^w_\varepsilon)\to\mu_k$ as $\varepsilon\downarrow0$ (Figure~\ref{fig:compression}), and, provided $H^1(G^w_\varepsilon)\hookrightarrow L^2(G^w_\varepsilon)$ is compact, Corollary~\ref{cor:crosscutSimple} shows that its first positive eigenvalue is simple whenever
\begin{equation}\label{eq:thinSimple}
 \mu_1\Bigl(1+\varepsilon^2\bigl(W_1^2/3+(4W_0^2/\pi^2)\,\mu_2\bigr)\Bigr)<\mu_2 .
\end{equation}
\end{corollary}
\begin{proof}
This is Proposition~\ref{prop:crosscut} with $\kappa_*^2=\varepsilon^2W_1^2/3$ and $c=(2\varepsilon W_0/\pi)^2$.
\end{proof}

\begin{figure}[ht]
\centering
\begin{tikzpicture}
\begin{scope}[x=3.4cm,y=1.7cm]
\filldraw[fill=gray!20,draw=black,thick] (0,0) -- plot[domain=0:1,samples=40] ({\x},{\x*\x}) -- (1,-1) -- plot[domain=1:0,samples=40] ({\x},{-\x*\x}) -- cycle;
\draw[dashed,gray] (-0.05,0) -- (1.05,0);
\end{scope}
\begin{scope}[shift={(7.2cm,0cm)},x=3.4cm,y=1.7cm]
\filldraw[fill=gray!20,draw=black,thick] (0,0) -- plot[domain=0:1,samples=40] ({\x},{0.25*\x*\x}) -- (1,-0.25) -- plot[domain=1:0,samples=40] ({\x},{-0.25*\x*\x}) -- cycle;
\draw[dashed,gray] (-0.05,0) -- (1.05,0);
\end{scope}
\draw[->,thick] (4.0cm,0cm) -- (6.6cm,0cm) node[midway,above,font=\small] {$(x,z)\mapsto(x,\varepsilon z)$};
\node[below,font=\small] at (1.7cm,-1.8cm) {$G^w_1$, $w(x)=x^{2}$};
\node[below,font=\small] at (8.9cm,-1.8cm) {$G^w_\varepsilon$, drawn with $\varepsilon=1/4$};
\end{tikzpicture}
\caption{The vertical compression $G^w_1\to G^w_\varepsilon$ for $w(x)=x^2$. As $\varepsilon\downarrow0$ the domain collapses onto the segment $(0,1)$, and by Corollary~\ref{cor:thin} each fixed-index Neumann eigenvalue converges to the corresponding weighted one-dimensional eigenvalue.}
\label{fig:compression}
\end{figure}
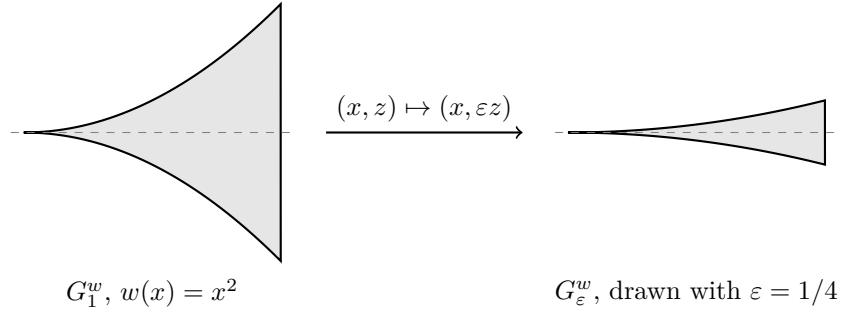

\begin{lemma}[Simplicity of the longitudinal spectrum]\label{lem:regularEndpointSimple}
Suppose in addition that $w$ is positive and $C^1$ on $(0,1]$ and that $H^1(G^w_1)\hookrightarrow L^2(G^w_1)$ is compact. Then the longitudinal eigenvalues satisfy $0=\mu_0<\mu_1<\mu_2<\cdots$, that is, every longitudinal eigenvalue is simple.
\end{lemma}
\begin{proof}
By the remark after Definition~\ref{def:crosscut}, the longitudinal problem has compact resolvent, so the $\mu_k$ are eigenvalues. An eigenfunction $f$ satisfies $-(wf')'=\mu wf$ weakly in $(0,1)$. By local regularity it is a classical $C^2$ solution. Near $x=1$ one has $w\ge c>0$, so $f\in H^1(1-\theta,1)$ for some $\theta>0$, and the equation $(wf')'=-\mu wf\in L^2(1-\theta,1)$ gives $wf'\in H^1(1-\theta,1)\subset C[1-\theta,1]$. Hence $f\in C^1[1-\theta,1]$. Testing the weak equation against functions that are smooth up to $x=1$ and vanish near $0$, and integrating by parts, gives the natural condition $f'(1)=0$, because $w(1)>0$. Since $w\in C^1$ and $w>0$ near $x=1$, the ODE is regular there, so $(f(1),f'(1))$ determines $f$. The solutions with $f'(1)=0$ therefore form a space of dimension at most one, and every eigenvalue is simple. Constants give the simple eigenvalue $0$, and $\mu_1>0$ since $(0,1)$ is connected.
\end{proof}

\begin{theorem}[Simplicity for thin profiles]\label{thm:thinSimple}
Under the hypotheses of Lemma~\ref{lem:regularEndpointSimple}, there is $\varepsilon_0>0$ such that for every $0<\varepsilon<\varepsilon_0$ the domain $G^w_\varepsilon$ is bounded and simply connected, has compact embedding $H^1(G^w_\varepsilon)\hookrightarrow L^2(G^w_\varepsilon)$, and has a simple first positive Neumann eigenvalue. One may take any $\varepsilon_0$ for which \eqref{eq:thinSimple} holds on $(0,\varepsilon_0)$.
\end{theorem}
\begin{proof}
Compactness for $G^w_\varepsilon$ follows from that of $G^w_1$, because $G^w_\varepsilon$ is the image of $G^w_1$ under the linear map $(x,y)\mapsto(x,\varepsilon y)$. By Lemma~\ref{lem:regularEndpointSimple}, $\mu_1<\mu_2$, so \eqref{eq:thinSimple} holds for all sufficiently small $\varepsilon$. Corollary~\ref{cor:thin} then gives $\lambda_1(G^w_\varepsilon)<\lambda_2(G^w_\varepsilon)$.
\end{proof}

\subsection{Example I: Lipschitz tubes}
\begin{proposition}\label{prop:tubes}
Let $h:[0,2]\to\R$ be Lipschitz with $\kappa=\|h'\|_\infty<\sqrt2$. Then
\[
 T_h=\{(x,y):0<x<2,\ h(x)<y<1+h(x)\}
\]
is a bounded simply connected Lipschitz domain, with a polynomial Neumann heat-kernel bound $Ct^{-1}$ and compact embedding, and its first positive Neumann eigenvalue is simple. More precisely,
\[
 \lambda_1(T_h)\le\frac{\pi^2}4<\frac{\pi^2}{2+\kappa^2}\le\lambda_2(T_h).
\]
\end{proposition}
\begin{proof}
$T_h$ is the cross-cut domain with $a=h$, $b=1+h$, $L=2$, so $\ell\equiv1$, $\ell_*=1$ and $\kappa_*=\kappa$. The weighted problem is the Neumann problem for $-f''$ on $(0,2)$, with $\mu_k=(k\pi/2)^2$. Thus $\mu_1=\pi^2/4$ and $\mu_2=\pi^2$. With $c=1/\pi^2$, Proposition~\ref{prop:crosscut} gives
\[
 \lambda_1\le\frac{\pi^2}4,\qquad
 \lambda_2\ge\frac{\pi^2}{1+\kappa^2+c\pi^2}=\frac{\pi^2}{2+\kappa^2},
\]
and $\pi^2/(2+\kappa^2)>\pi^2/4$ if and only if $\kappa^2<2$. The domain $T_h$ is Lipschitz for every finite $\kappa$: at each of the four corners a vertical side meets a graph of slope at most $\kappa$ transversally, and in a suitably rotated frame the boundary near the corner is a Lipschitz graph. Hence the heat-kernel bound $Ct^{-1}$ and compactness follow from Section~\ref{sec:examples}. The restriction $\kappa<\sqrt2$ is used only for the spectral gap.
\end{proof}
For $h(x)=|x-1|/10$ one obtains a nonconvex polygon with a reentrant corner at $(1,1)$ (Figure~\ref{fig:examples}(a)). Slope bounds strictly below $\sqrt2$ are allowed, and no differentiability of $h$ beyond Lipschitz continuity is required.

\subsection{Example II: quadratic cusps}
\begin{theorem}\label{thm:quadCusp}
For every $0<\varepsilon\le1/5$, the domain
\[
 Q_\varepsilon=\{(x,y):0<x<1,\ |y|<\varepsilon x^2\}
\]
is bounded and simply connected, has a non-Lipschitz cusp (of quasihyperbolic order $1/2$ in the sense of Appendix~\ref{app:classes}), satisfies $p_{Q_\varepsilon}(t,z,w)\le C_\varepsilon t^{-3/2}$ for $0<t\le1$, has compact embedding, and has a simple first positive Neumann eigenvalue. More precisely,
\[
 \lambda_1(Q_\varepsilon)<22.21<23.3<\lambda_2(Q_\varepsilon).
\]
\end{theorem}
\begin{proof}
The geometric statements and the heat-kernel bound are those of the power cusp $\Om_\alpha$ with $\alpha=2$ in Section~\ref{subsec:powerCusps}, since $Q_\varepsilon$ is the image of $\Om_2$ under the bi-Lipschitz map $(x,y)\mapsto(x,\varepsilon y)$. Compactness follows from Remark~\ref{rem:adds}. For the spectrum, $Q_\varepsilon=G^w_\varepsilon$ with $w(x)=x^2$, $W_0=1$ and $W_1=2$.

\emph{The longitudinal problem.} It is the problem for $\int_0^1x^2|f'|^2$ relative to $\int_0^1x^2f^2$, that is, the radial Neumann problem on the unit ball $\mathbb B\subset\R^3$: for $\Psi(z)=f(|z|)$ one has $\int_{\mathbb B}|\nabla \Psi|^2=4\pi\int_0^1 x^2f'^2\,dx$ and $\int_{\mathbb B}\Psi^2=4\pi\int_0^1 x^2f^2\,dx$. By Lemma~\ref{lem:regularEndpointSimple} its eigenvalues are simple. An eigenfunction solves $-(x^2f')'=\mu x^2f$ with $f'(1)=0$. For $\mu=k^2>0$ the solutions are $A\sin(kx)/x+B\cos(kx)/x$, and finiteness of $\int x^2|f'|^2$ forces $B=0$. The condition $f'(1)=0$ for $f=\sin(kx)/x$ reads $k\cos k=\sin k$, that is, $\tan k=k$. The function $\tan k-k$ increases from $-\infty$ to $+\infty$ on each interval $(j\pi-\pi/2,\,j\pi+\pi/2)$ and is negative on $(j\pi-\pi/2,j\pi]$ for $j\ge1$, so the positive roots are $k_j\in(j\pi,j\pi+\pi/2)$, $j=1,2,\dots$. Hence
\[
 \mu_1=k_1^2<(3\pi/2)^2<22.21,\qquad \mu_2=k_2^2>(2\pi)^2>39.47 .
\]

\emph{The comparison.} By Corollary~\ref{cor:thin}, with $W_1^2/3=4/3$ and $4W_0^2/\pi^2=4/\pi^2$,
\[
 \lambda_1\le\mu_1<22.21,\qquad
 \lambda_2\ge\frac{\mu_2}{1+\varepsilon^2\bigl(4/3+(4/\pi^2)\mu_2\bigr)}\ \ge\ \frac{4\pi^2}{1+\varepsilon^2(4/3+16)}=\frac{4\pi^2}{1+(52/3)\varepsilon^2},
\]
where we used that the left-hand function is increasing in $\mu_2$ and inserted $\mu_2>4\pi^2$. For $\varepsilon\le1/5$ the right-hand side is at least $4\pi^2/1.6934>23.3$. Hence $\lambda_1<22.21<23.3<\lambda_2$.
\end{proof}
Numerically, the roots of $\tan k=k$ give $\mu_1\approx20.19$ and $\mu_2\approx59.68$. These values are not used in the proof.

\subsection{Example III: an exponential cusp without a polynomial heat-kernel bound}\label{sec:explicitCusp}

\begin{theorem}\label{thm:explicitExpCusp}
For every $0<\varepsilon\le1/10$, the domain
\[
 E_\varepsilon=\{(x,y):0<x<1,\ |y|<\varepsilon e^{-1/x}\}
\]
is bounded and simply connected, has a non-Lipschitz cusp, and has compact embedding $H^1(E_\varepsilon)\hookrightarrow L^2(E_\varepsilon)$. Its Neumann eigenvalues, counted with multiplicity from $\lambda_0=0$, satisfy
\[
 \lambda_1(E_\varepsilon)<31.3<32<\lambda_2(E_\varepsilon).
\]
In particular, its first positive Neumann eigenvalue is simple. Nevertheless, there are no constants $C,N<\infty$ such that $p_{E_\varepsilon}(t,z,w)\le Ct^{-N/2}$ for $0<t\le1$ and almost every $z,w$.
\end{theorem}

The proof combines a longitudinal spectral estimate with the cross-cut comparison. Compactness follows from the continuous-graph boundary criterion, while test functions near the tip rule out every polynomial heat-kernel bound.

\subsubsection{Geometry and compact embedding}
\paragraph{Geometry.} Near the tip, the domain lies to the right of the continuous graph
\[
 x=g_\varepsilon(y),\qquad
 g_\varepsilon(y)=\frac1{\log(\varepsilon/|y|)}\ (y\ne0),
 \qquad g_\varepsilon(0)=0,
\]
and away from the tip the boundary is Lipschitz. The opening angle at the tip tends to zero, so the cusp is not Lipschitz. The map $(x,t)\mapsto(x,\varepsilon e^{-1/x}t)$ is a homeomorphism of $(0,1)\times(-1,1)$ onto $E_\varepsilon$, hence $E_\varepsilon$ is bounded and simply connected.

\paragraph{Compact embedding.} $E_\varepsilon$ is a bounded domain whose boundary is locally the graph of a continuous function, and for such domains the embedding $H^1\hookrightarrow L^2$ is compact. See Burenkov \cite{Burenkov} and the introduction of \cite{BD}.

\subsubsection{The longitudinal spectral gap}
Set $w(x)=e^{-1/x}$ and consider the closed form
\[
 a[f]=\int_0^1 w|f'|^2\,dx\quad\text{in }L^2((0,1),w\,dx),
\]
with domain $\mathcal V=\{f\in H^1_{\mathrm{loc}}(0,1):\int_0^1w(|f|^2+|f'|^2)\,dx<\infty\}$. This is the longitudinal problem of Section~\ref{subsec:crosscutdef} for the profile $w$. Its compactness follows from the compact embedding for $E_1$ just established. Denote its eigenvalues by $0=\mu_0\le\mu_1\le\mu_2\le\cdots$.

\begin{lemma}\label{lem:expLongitudinal}
For the profile $w(x)=e^{-1/x}$, the longitudinal eigenvalues satisfy
\[
 \mu_1<31.3,\qquad \mu_2>4\pi^2-\frac{27}{4}>32.7 .
\]
\end{lemma}
\begin{proof}
\emph{Upper bound.} Write $I_j=\int_0^1x^je^{-1/x}\,dx$. The trial function $f(x)=x-I_1/I_0$ is orthogonal to constants in $L^2(w\,dx)$, and $\int_0^1w|f'|^2=I_0$, $\int_0^1wf^2=I_2-I_1^2/I_0$, so
\[
 \mu_1\le\frac{I_0}{I_2-I_1^2/I_0}.
\]
The following enclosures suffice:
\[
 0.14840<I_0<0.1487,\qquad
 0.1096<I_1<0.1098,\qquad
 0.0860<I_2<0.0861.
\]
They imply
\[
 \mu_1<\frac{0.1487}{0.0860-(0.1098)^2/0.14840}<31.3 .
\]
For a reproducible evaluation of the enclosures, integration by parts gives
\[
 I_0=e^{-1}-E_1(1),\qquad I_1=\tfrac12E_1(1),\qquad
 I_2=\tfrac13(e^{-1}-I_1),
\]
where
\[
 E_1(1)=\int_1^\infty\frac{e^{-t}}{t}\,dt=-\gamma-\sum_{m=1}^{\infty}\frac{(-1)^m}{m\,m!}
\]
and $\gamma$ is the Euler--Mascheroni constant. Only the bounds
\[
 0.5772<\gamma<0.5773,\qquad 0.3678<e^{-1}<0.3679
\]
are needed. To check the enclosures using rational arithmetic, put
\[
 S_{10}=-\sum_{m=1}^{10}\frac{(-1)^m}{m\,m!},
 \qquad R=\frac{1}{11\cdot11!}.
\]
The alternating-series estimate gives
\[
 S_{10}-0.5773<E_1(1)<S_{10}+R-0.5772.
\]
Substitution into the formulas for $I_0,I_1,I_2$, together with the displayed bounds for $e^{-1}$, yields the enclosures above. The tightest of them is the lower bound $I_0>0.3678-(S_{10}+R-0.5772)>0.14840$.

\emph{Transformation.} For the lower bound, put $g=\sqrt w\,f$. This map is unitary from $L^2(w\,dx)$ to $L^2(dx)$. Since $w'/w=x^{-2}$, one has $f'=w^{-1/2}\bigl(g'-g/(2x^2)\bigr)$, so for real $g$ smooth up to $x=1$ and vanishing near $0$,
\[
 w|f'|^2=\Bigl|g'-\frac{g}{2x^2}\Bigr|^2=|g'|^2-\frac{gg'}{x^2}+\frac{|g|^2}{4x^4},
 \qquad
 \int_0^1\frac{gg'}{x^2}\,dx=\frac{|g(1)|^2}{2}+\int_0^1\frac{|g|^2}{x^3}\,dx,
\]
the second identity by integration by parts, since $(1/(2x^2))'=-x^{-3}$ and the boundary term at $0$ vanishes. Integrating the first identity gives
\begin{equation}\label{eq:expSchrodinger}
 a[f]=\int_0^1\bigl(|g'|^2+V|g|^2\bigr)\,dx
       -\frac12|g(1)|^2,
 \qquad V(x)=\frac1{4x^4}-\frac1{x^3}\ge-\frac{27}{4}.
\end{equation}
The last bound follows by minimizing $t^4/4-t^3$ over $t=1/x\ge1$. The minimum occurs at $t=3$.

\emph{Form domain.} We justify \eqref{eq:expSchrodinger}, as an inequality, on the whole form domain. First, truncating $f\in\mathcal V$ to bounded functions gives convergence in the weighted form norm. A bounded function can then be cut off on $(0,r)$, with transition on $(r,2r)$: the added derivative cost is bounded by $C\|f\|_\infty^2r^{-2}\int_r^{2r}w(x)\,dx\to0$, and the contribution of $(0,2r)$ to $\int w(|f|^2+|f'|^2)$ tends to zero by integrability. Smooth approximation away from zero therefore produces a form core of functions that are smooth up to $x=1$ and vanish near zero, on which \eqref{eq:expSchrodinger} holds. The trace inequality
\[
 |g(1)|^2\le\eta\|g'\|_2^2+C_\eta\|g\|_2^2
\]
and the lower bound for $V$ show that the weighted form norm controls $\|g\|_{H^1}$. Passing to the closure, every transformed form-domain function belongs to $H^1(0,1)$ with $g(0)=0$, and
\[
 a[f]\ge\int_0^1|g'|^2\,dx-\tfrac12|g(1)|^2
                         -\tfrac{27}{4}\|g\|_2^2.
\]
This inequality is all that is needed. No additional endpoint condition at zero is imposed on $f$.

\emph{Lower bound.} The comparison form $b[g]=\int_0^1|g'|^2-\tfrac12|g(1)|^2$ on $\{g\in H^1(0,1):g(0)=0\}$ has the boundary conditions $g(0)=0$, $g'(1)=g(1)/2$. It is strictly positive, since $|g(1)|^2\le\int_0^1|g'|^2$. Write its eigenvalues as $0<\nu_0<\nu_1<\cdots$. Its eigenfunctions are $\sin(kx)$ with $\tan k=2k$ and eigenvalues $\nu_j=k_j^2$, where $k_j$ is the unique root of $\tan k=2k$ in $(j\pi,j\pi+\pi/2)$, $j=0,1,2,\ldots$. To verify uniqueness, note that $\tan k-2k$ is strictly convex on each $(j\pi,j\pi+\pi/2)$. For $j\ge1$ it is negative at the left endpoint and tends to $+\infty$ at the right endpoint, so it has exactly one root. For $j=0$ it vanishes at zero and has derivative $-1$ there, so it again has exactly one root in the open interval. The value $k=0$ yields no nonzero eigenfunction. On $(j\pi-\pi/2,j\pi)$ it is negative for $j\ge1$. In particular $\nu_2>4\pi^2$. The min--max principle, which remains valid for the possibly smaller transformed form domain, gives $\mu_2\ge\nu_2-27/4>4\pi^2-27/4$.
\end{proof}
Numerically, $\mu_1\approx24.37$ and $\mu_2\approx76.19$. These values are not used in the proof.

\subsubsection{Proof of Theorem~\ref{thm:explicitExpCusp}}
\begin{proof}[Proof of Theorem~\ref{thm:explicitExpCusp}]
\emph{Step 1: no polynomial heat-kernel bound.} Fix $\varepsilon>0$. Let $u_r$ depend only on $x$, equal one for $x\le r$, decrease linearly to zero on $(r,r+r^2)$, and vanish beyond. The substitution $s=1/x$ gives
\begin{equation}\label{eq:tipMass}
 \int_0^r e^{-1/x}\,dx=\int_{1/r}^{\infty}e^{-s}s^{-2}\,ds
 \asymp r^2e^{-1/r}.
\end{equation}
The upper bound follows from $s^{-2}\le r^2$ on $s\ge1/r$, and the lower bound from integration over $[1/r,1/r+1]$. On $(r,r+r^2)$ the weight $e^{-1/x}$ lies between $e^{-1/r}$ and $e\,e^{-1/r}$, and the squared gradient of $u_r$ equals $r^{-4}$ on this strip, whose area is $O_\varepsilon(r^2e^{-1/r})$. Thus, for every $q>2$,
\[
 \|u_r\|_{L^q(E_\varepsilon)}^2
 \ge c_{\varepsilon,q}\,r^{4/q}e^{-2/(qr)},\qquad
 \|u_r\|_{H^1(E_\varepsilon)}^2
 \le C_\varepsilon\, r^{-2}e^{-1/r},
\]
and therefore
\[
 \frac{\|u_r\|_q^2}{\|u_r\|_{H^1}^2}
 \ge c_{\varepsilon,q}\,r^{2+4/q}
       e^{(1-2/q)/r}\longrightarrow\infty\qquad(r\downarrow0).
\]
Hence no continuous embedding of $H^1(E_\varepsilon)$ into any $L^q$, $q>2$, holds, so \eqref{eq:examplesSob} fails for every $q>2$. By the equivalence recorded in Section~\ref{sec:examples} (see \cite[Proposition 3]{BD}) and by \eqref{eq:standardSob}, every polynomial heat-kernel bound with $N>2$ would imply \eqref{eq:examplesSob}, so all such bounds are excluded. A bound with exponent at most two would imply one with a larger exponent and is excluded as well.

\emph{Step 2: the upper bound $\lambda_1(E_\varepsilon)<31.3$.} The domain $E_\varepsilon$ is the thin-profile domain $G^w_\varepsilon$ with $w(x)=e^{-1/x}$, $W_0=e^{-1}$ and $W_1=\sup_{0<x<1}e^{-1/x}x^{-2}=4e^{-2}$ (attained at $x=1/2$). By the upper bound in \eqref{eq:thin} and Lemma~\ref{lem:expLongitudinal}, $\lambda_1(E_\varepsilon)\le\mu_1<31.3$.

\emph{Step 3: the lower bound $\lambda_2(E_\varepsilon)>32$.}
The profile constants satisfy
\[
 \frac{W_1^2}{3}=\frac{16e^{-4}}3<0.1,
 \qquad
 \frac{4W_0^2}{\pi^2}=\frac{4}{e^2\pi^2}<0.055.
\]
For fixed $\varepsilon>0$, the function
\[
 \mu\longmapsto
 \frac{\mu}{1+\varepsilon^2\bigl(W_1^2/3+(4W_0^2/\pi^2)\mu\bigr)}
\]
is increasing. Hence \eqref{eq:thin} and Lemma~\ref{lem:expLongitudinal} give
\[
 \lambda_2(E_\varepsilon)>
 \frac{32.7}{1+\varepsilon^2(0.1+0.055\cdot32.7)}.
\]
If $0<\varepsilon\le1/10$, the denominator is less than $1.02$, so
\[
 \lambda_2(E_\varepsilon)>\frac{32.7}{1.02}>32.
\]
Combining this with Step~2 gives
\[
 \lambda_1(E_\varepsilon)<31.3<32<\lambda_2(E_\varepsilon),
 \qquad 0<\varepsilon\le1/10,
\]
which establishes simplicity and completes the proof.
\end{proof}

\begin{remark}[Scope of the numerical bounds]\label{rem:numerics}
The proofs of Theorems~\ref{thm:quadCusp} and~\ref{thm:explicitExpCusp} use explicit scalar enclosures, not numerical eigenvalue approximations. The ranges $\varepsilon\le1/5$ and $\varepsilon\le1/10$ are chosen to keep the enclosures and the final spectral gaps simple. They are sufficient and are not asserted to be optimal.
\end{remark}

\subsection{Summary of the examples}\label{subsec:summary}
Table~\ref{tab:examples} collects the three families and the bounds proved for them.

\begin{table}[ht]
\centering\small
\begin{tabular}{@{}p{0.28\textwidth}p{0.25\textwidth}p{0.39\textwidth}@{}}
\hline
Domain and range & Heat-kernel bound & Proved spectral gap\\
\hline
$T_h$, $\kappa<\sqrt2$ & $Ct^{-1}$ & $\lambda_1\le\pi^2/4<\pi^2/(2+\kappa^2)\le\lambda_2$\\[3pt]
$Q_\varepsilon$, $0<\varepsilon\le1/5$ & $C_\varepsilon t^{-3/2}$ & $\lambda_1<22.21<23.3<\lambda_2$\\[3pt]
$E_\varepsilon$, $0<\varepsilon\le1/10$ & No polynomial bound & $\lambda_1<31.3<32<\lambda_2$\\
\hline
\end{tabular}
\caption{The explicit examples of Section~\ref{sec:simplicity}. All three are bounded, simply connected, have compact Sobolev embedding, and have a simple first positive Neumann eigenvalue.}
\label{tab:examples}
\end{table}

\begin{corollary}\label{cor:allExamples}
Assume Hypothesis~\ref{hyp:smooth}. The hot-spots conclusions \eqref{eq:HS} and, for eigenfunctions continuous up to the boundary, \eqref{eq:boundary} hold for the Lipschitz tubes $T_h$ with $\|h'\|_\infty<\sqrt2$, for the quadratic cusps $Q_\varepsilon$ with $0<\varepsilon\le1/5$, and for the exponential cusps $E_\varepsilon$ with $0<\varepsilon\le1/10$, for example for
\[
 \{(x,y):0<x<1,\ |y|<10^{-2}e^{-1/x}\}.
\]
The last family satisfies compact embedding and simplicity but lies outside every class defined by a polynomial Neumann heat-kernel upper bound. Hence for it Theorem~\ref{thm:main} applies, whereas Corollary~\ref{cor:HKhot} does not.
\end{corollary}
\begin{proof}
Each family consists of bounded simply connected domains with compact embedding and a simple first positive eigenvalue (Proposition~\ref{prop:tubes} and Theorems~\ref{thm:quadCusp} and~\ref{thm:explicitExpCusp}), so Theorem~\ref{thm:main} applies.
\end{proof}

\section{Unbounded domains of finite area}\label{sec:unboundedFiniteArea}

Theorem~\ref{thm:main} allows unbounded domains of finite area. This permits examples whose boundary recedes to infinity while the cross-cuts shrink rapidly. The Gaussian horn below has compact Sobolev embedding and, for an explicit range of widths, a simple first positive Neumann eigenvalue. Sufficiently thin horns with profile $e^{-x^\alpha}$, $\alpha\ge2$, have the same properties (Remark~\ref{rem:otherHorns}). None of these domains admits a polynomial Neumann heat-kernel bound (Remark~\ref{rem:hornNoPolynomial}). The final subsection gives a different unbounded example, consisting of square rooms joined by rectangular passages, for which compactness and simplicity follow from elementary estimates.

\subsection{The Gaussian horn}
Figure~\ref{fig:gaussian-horn} shows the domain considered in the following theorem.

\begin{theorem}[The Gaussian horn]\label{thm:GaussianHorn}
For $\varepsilon>0$, let
\[
 \Horn_\varepsilon=\{(x,y):x>0,\ |y|<\varepsilon e^{-x^2}\}.
\]
Then $\Horn_\varepsilon$ is unbounded and simply connected,
$|\Horn_\varepsilon|=\varepsilon\sqrt\pi$, and
$H^1(\Horn_\varepsilon)\hookrightarrow L^2(\Horn_\varepsilon)$ is compact.
Its first two positive Neumann eigenvalues satisfy
\begin{equation}\label{eq:GaussianGap}
 \lambda_1(\Horn_\varepsilon)\le4,
 \qquad
 \lambda_2(\Horn_\varepsilon)
 \ge\frac{8}{1+\varepsilon^2\bigl(2/(3e)+32/\pi^2\bigr)}.
\end{equation}
In particular, $\lambda_1(\Horn_\varepsilon)$ is simple whenever $0<\varepsilon\le\tfrac12$.
\end{theorem}
\begin{figure}[htbp]
\centering
\begin{tikzpicture}[x=1cm,y=1cm,font=\small,line cap=round,line join=round]
\fill[black!8] (0,-1) -- (0,1)
 plot[domain=0:2.8,samples=100] ({2.65*\x},{exp(-\x*\x)})
 -- plot[domain=2.8:0,samples=100] ({2.65*\x},{-exp(-\x*\x)}) -- cycle;
\draw[thick] (0,-1)--(0,1);
\draw[thick] plot[domain=0:2.25,samples=90] ({2.65*\x},{exp(-\x*\x)});
\draw[thick] plot[domain=0:2.25,samples=90] ({2.65*\x},{-exp(-\x*\x)});
\draw[thick,densely dashed] plot[domain=2.25:2.8,samples=30] ({2.65*\x},{exp(-\x*\x)});
\draw[thick,densely dashed] plot[domain=2.25:2.8,samples=30] ({2.65*\x},{-exp(-\x*\x)});
\draw[->,black!50] (-.2,0)--(8.3,0) node[right,text=black] {$x$};
\node[below left] at (0,0) {$0$};
\node[above left] at (0,1) {$\varepsilon$};
\node[below left] at (0,-1) {$-\varepsilon$};
\node at (1.15,0.05) {$\mathcal G_\varepsilon$};
\node[anchor=west] at (3.1,1.05) {$y=\varepsilon e^{-x^2}$};
\draw[thin] (3.2,.85)--(2.65,.368);
\node[anchor=west] at (3.1,-1.05) {$y=-\varepsilon e^{-x^2}$};
\draw[thin] (3.2,-.85)--(2.65,-.368);
\node at (7.75,.55) {$x\to\infty$};
\end{tikzpicture}
\caption{The Gaussian horn $\mathcal G_\varepsilon=\{(x,y):x>0,\ |y|<\varepsilon e^{-x^2}\}$. Dashed curves indicate continuation to infinity.}
\label{fig:gaussian-horn}
\end{figure}
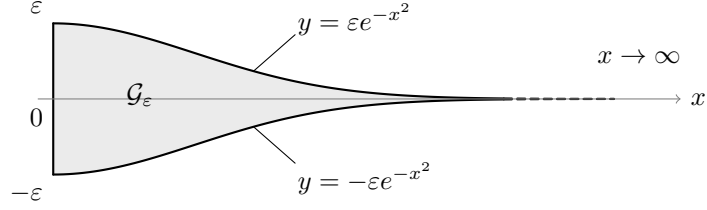

\begin{proof}[Proof of Theorem~\ref{thm:GaussianHorn}]
Write $w(x)=e^{-x^2}$. Note that $w$ is decreasing on $(0,\infty)$. The map
$(x,t)\mapsto(x,\varepsilon w(x)t)$ is a homeomorphism from
$(0,\infty)\times(-1,1)$ onto $\Horn_\varepsilon$, proving simple
connectivity. Integration gives the stated area.

\emph{Cross-cut estimates.}
For $u\in H^1(\Horn_\varepsilon)$, define its cross-cut average, as in Section~\ref{sec:simplicity}, by
\[
 \bar u(x)=\frac12\int_{-1}^1u(x,\varepsilon w(x)t)\,dt.
\]
Put $A=\int_{\Horn_\varepsilon}|u_x|^2$, $B=\int_{\Horn_\varepsilon}|u_y|^2$,
$W_0=\sup w=1$, and $W_1=\sup|w'|=\sqrt{2/e}$.
Jensen's inequality gives \eqref{eq:hornAvgNorm}, and differentiation of the average with the Cauchy--Schwarz inequality gives \eqref{eq:hornAvgEnergy}:
\begin{align}
 2\varepsilon\int_0^\infty w|\bar u|^2\,dx
     &\le\|u\|_2^2,\label{eq:hornAvgNorm}\\
 2\varepsilon\int_0^\infty w|\bar u'|^2\,dx
     &\le\left(\sqrt A+
           \frac{\varepsilon W_1}{\sqrt3}\sqrt B\right)^2
      \le(1+\varepsilon^2W_1^2/3)(A+B).
       \label{eq:hornAvgEnergy}
\end{align}
The constant $1/3$ is the average of $t^2$ on $(-1,1)$.
The Neumann Poincar\'e inequality on each cross-cut, whose length $2\varepsilon w(x)$ is at most $2\varepsilon w(R)$ for $x>R$ because $w$ is decreasing, gives
\begin{equation}\label{eq:hornFluctuation}
 \int_{\{x>R\}\cap \Horn_\varepsilon}|u-\bar u(x)|^2
 \le\frac{4\varepsilon^2w(R)^2}{\pi^2}
       \int_{\{x>R\}\cap \Horn_\varepsilon}|u_y|^2,
 \qquad R\ge0.
\end{equation}
These statements follow first on finite truncations, where the change
of coordinates is regular, and then by Sobolev approximation and
monotone convergence. On each cross-cut, the average and fluctuation
are orthogonal in $L^2$.

\emph{Compact embedding.}
Let $\mathcal V_w=\{f\in H^1_{\mathrm{loc}}(0,\infty):\int_0^\infty w(|f|^2+|f'|^2)\,dx<\infty\}$ be the weighted form domain. The unitary substitution $g=\sqrt w\,f=e^{-x^2/2}f$ gives
\begin{equation}\label{eq:hornOscillatorForm}
 \int_0^\infty w|f'|^2\,dx
 =\int_0^\infty\bigl(|g'|^2+(x^2-1)|g|^2\bigr)\,dx.
\end{equation}
For smooth functions of bounded support this is integration by parts, since $w f'^2=(g'+xg)^2=g'^2+x(g^2)'+x^2g^2$ and the boundary term $xg^2$ vanishes at $0$ and beyond the support. Here $C_c^\infty([0,\infty))$ denotes functions smooth up to zero and of compact support in the half-line. No vanishing condition is imposed at zero. These functions are dense in $\mathcal V_w$ (multiply by $\chi(x/n)$ for a cutoff $\chi$ and mollify) and in $\{g\in H^1(0,\infty):xg\in L^2\}$, and the map $f\mapsto\sqrt w\,f$ carries the first space onto the second, so the identity extends to $\mathcal V_w$. Equivalently, the identity with $\int|g|^2$ added to both sides identifies the two completed form norms. In particular, since $x^2\ge R^2$ on $(R,\infty)$,
\[
 \int_R^\infty w|f|^2\,dx
 \le\frac1{R^2}\int_0^\infty x^2|g|^2\,dx
 \le\frac1{R^2}\int_0^\infty w(|f'|^2+|f|^2)\,dx.
\]
By orthogonality, \eqref{eq:hornAvgNorm}--\eqref{eq:hornFluctuation}
imply
\begin{equation}\label{eq:GaussianTail}
 \int_{\{x>R\}\cap \Horn_\varepsilon}|u|^2
 \le \frac{\|u\|_2^2+(1+2\varepsilon^2/(3e))\|\nabla u\|_2^2}{R^2}
       +\frac{4\varepsilon^2e^{-2R^2}}{\pi^2}\|u_y\|_2^2,\qquad R>0.
\end{equation}
Thus the $L^2$ tails tend to zero uniformly on bounded subsets of
$H^1(\Horn_\varepsilon)$. Every finite truncation is a bounded Lipschitz
domain. Local Rellich compactness, a diagonal subsequence, and
\eqref{eq:GaussianTail} prove global compactness.

\emph{The longitudinal spectrum.}
Consider the form $\int_0^\infty w|f'|^2$ in $L^2(w\,dx)$ with form domain $\mathcal V_w$. No trace restriction is imposed on the form domain at $x=0$. The associated operator has the natural Neumann condition there. Its differential expression is
\[
 Lf=-f''+2xf'.
\]
Under \eqref{eq:hornOscillatorForm}, it is the half-line harmonic
oscillator $-g''+(x^2-1)g$ with Neumann condition at $0$: neither form domain imposes a condition at $0$, and for smooth $f$ one has $g'(0)=f'(0)$ because $(\sqrt w)'(0)=0$.
Even reflection identifies this with the even subspace of the
full-line harmonic oscillator. The Hermite spectrum is therefore
\[
 \mu_k=4k,\qquad k=0,1,2,\ldots,
\]
with eigenfunctions the even Hermite polynomials. In particular, $1$, $x^2-\tfrac12$, and $x^4-3x^2+\tfrac34$ are the first three weighted eigenfunctions, as one checks by applying $L$.

\emph{Spectral comparison and simplicity.}
Here the cross-cut comparison can be justified directly on the
half-line. Put
\[
 a=1+\varepsilon^2W_1^2/3,
 \qquad c=4\varepsilon^2W_0^2/\pi^2.
\]
Trial functions depending only on $x$ give $\lambda_k\le\mu_k$.
For the lower bound, suppose first that $c\lambda_k<1$ and average
the span of the first $k+1$ planar eigenfunctions. For every $u$ in
that span, orthogonality and \eqref{eq:hornFluctuation} give
\[
 2\varepsilon\int w|\bar u|^2
 =\|u\|_2^2-\|u-\bar u\|_2^2
 \ge(1-c\lambda_k)\|u\|_2^2.
\]
The averaging map is consequently injective on that span.
Equation~\eqref{eq:hornAvgEnergy} and the weighted min--max principle
yield $\mu_k\le a\lambda_k/(1-c\lambda_k)$, or
\begin{equation}\label{eq:hornComparison}
 \frac{\mu_k}{1+\varepsilon^2(W_1^2/3+4W_0^2\mu_k/\pi^2)}
 \le\lambda_k\le\mu_k.
\end{equation}
If $c\lambda_k\ge1$, the same lower bound is immediate since its
left-hand side is less than $1/c$.
Substituting $\mu_1=4$, $\mu_2=8$ proves
\eqref{eq:GaussianGap}. Finally,
\[
 \frac{2}{3e}+\frac{32}{\pi^2}<4.
\]
For example, $e>2$ and $\pi>3$ bound the left-hand side by
$1/3+32/9=35/9<4$. For $\varepsilon\le1/2$ the denominator
in \eqref{eq:GaussianGap} is strictly less than $2$, so
$\lambda_1\le4<\lambda_2$.
\end{proof}

\begin{corollary}\label{cor:horn}
Assume Hypothesis~\ref{hyp:smooth}. For $0<\varepsilon\le\tfrac12$, every real first positive Neumann eigenfunction of $\Horn_\varepsilon$ satisfies \eqref{eq:HS}.
\end{corollary}
\begin{proof}
By Theorem~\ref{thm:GaussianHorn}, $\Horn_\varepsilon$ satisfies the hypotheses of Theorem~\ref{thm:main}.
\end{proof}

\begin{remark}[Extrema approached at infinity]\label{rem:hornVacuous}
For every $\delta>0$ the layer $B_\delta(\Horn_\varepsilon)$ contains the whole tail $\{x>R\}\cap\Horn_\varepsilon$ for $R$ sufficiently large, because the cross-cuts have length $2\varepsilon e^{-x^2}$. If a supremum or infimum is approached along that tail, the corresponding identity in \eqref{eq:HS} is automatic. No assertion about growth or boundedness of the planar eigenfunction at infinity follows from the longitudinal comparison used here.
\end{remark}

\begin{remark}[Other horns]\label{rem:otherHorns}
The same argument applies to
\[
 \Horn_{\varepsilon,\alpha}
 =\{(x,y):x>0,\ |y|<\varepsilon e^{-x^\alpha}\},
 \qquad \alpha\ge2.
\]
The longitudinal substitution $g=e^{-x^\alpha/2}f$ gives the potential
\[
 V_\alpha(x)=\frac{\alpha^2}{4}x^{2\alpha-2}
             -\frac{\alpha(\alpha-1)}2x^{\alpha-2},
\]
which is bounded below and tends to infinity (for $1<\alpha<2$ it would tend to $-\infty$ at $0$, which is why $\alpha\ge2$ is assumed). The resulting weighted
tail estimate and the transverse Poincar\'e inequality give compact
embedding. The longitudinal eigenvalues are discrete, because $V_\alpha\to\infty$ gives a compact resolvent, and simple, because the equation is regular at $0$ for $\alpha\ge2$, so that the Neumann condition $g'(0)=0$ determines the eigenfunction up to a scalar.
For each fixed $\alpha$, the comparison in \eqref{eq:hornComparison},
with $W_0=1$ and $W_1=\sup|(e^{-x^\alpha})'|<\infty$, therefore gives
$\lambda_1<\lambda_2$ for all sufficiently small $\varepsilon>0$.
These horns also satisfy the hypotheses of Theorem~\ref{thm:main}.
\end{remark}

\begin{remark}[Failure of polynomial heat-kernel bounds]
\label{rem:hornNoPolynomial}
On any unbounded open set $D$ of finite area, the continuous embedding
$H^1(D)\hookrightarrow L^q(D)$ fails for every $q>2$. See Adams \cite{Adams1973}.
By the Sobolev--heat-kernel equivalence
\cite[Proposition~3]{BD}, this excludes every spatially uniform bound
$p_D(t,z,z')\le Ct^{-N/2}$ for $0<t\le1$ with finite $N>0$ (a bound with $N\le2$ would imply one with $N>2$ on $0<t\le1$).
For the horns the failure can be seen directly. On $\Horn_{\varepsilon,\alpha}$, the function $u=e^{sx^\alpha}$ has
\[
 \int u^2=2\varepsilon\int_0^\infty e^{(2s-1)x^\alpha}\,dx,\qquad
 \int|\nabla u|^2=2\varepsilon s^2\alpha^2\int_0^\infty x^{2\alpha-2}e^{(2s-1)x^\alpha}\,dx,
\]
both finite for $s<1/2$, while $\int u^q=2\varepsilon\int_0^\infty e^{(qs-1)x^\alpha}\,dx=\infty$ for $s\ge1/q$. For $q>2$ take $s=1/q$.
This does not preclude compact embedding into $L^2(D)$.
The horns therefore illustrate the greater scope of the compactness hypothesis in Theorem~\ref{thm:main}.
\end{remark}

\subsection{An unbounded rooms-and-passages domain}
\label{subsec:rooms-passages}
The following construction is motivated by the example of Ba\~{n}uelos and Davis \cite[Section~4]{BanuelosDavis1989}. It consists of shrinking disks joined by thin passages. The example distinguishes convergence of the normalized Dirichlet heat kernel that is uniform in one spatial variable from convergence that is uniform in both variables. This distinction is related to intrinsic ultracontractivity (IU) and Half-IU. Here we replace the disks with squares and choose the dimensions to ensure compactness of the Neumann resolvent and simplicity of the first positive eigenvalue. Unlike the Gaussian horn, this domain has repeated bottlenecks and infinitely many corners, although every finite truncation is Lipschitz. An additional advantage is that both spectral properties follow from elementary estimates on squares and rectangles, which simplify the analysis compared with the disks in the original construction.

Rooms-and-passages domains also appear in the analysis of Neumann essential spectra \cite{HSS} and comparisons with one-dimensional operators on their skeletons \cite{BEW}. The probabilistic background includes exit times and lifetimes of conditioned Brownian motion \cite{CM,BanuelosDavis1989}, and their connections with intrinsic ultracontractivity \cite{DS,BanuelosIU}. These connections motivate the geometry. No assertion about Dirichlet intrinsic ultracontractivity is needed below.

Let $L_n$ denote passage length and $h_n$ passage width. These differ from the cross-cut width $\ell$ in Section~\ref{sec:simplicity}. For $n\ge1$, set
\begin{equation}\label{eq:rp-parameters}
 r_n=e^{-n^2},\qquad L_n=\frac1n,\qquad
 h_n=n^2e^{-2(n+1)^2}=n^2r_{n+1}^2.
\end{equation}
Define $x_1=0$ and $x_{n+1}=x_n+r_n+L_n$, and let
\begin{align*}
 Q_n&=(x_n,x_n+r_n)\times(-r_n/2,r_n/2),\\
 P_n&=(x_n+r_n,x_{n+1})\times(-h_n/2,h_n/2).
\end{align*}
The domain is
\begin{equation}\label{eq:rp-domain}
 \Om=\operatorname{int}\!\left(\bigcup_{n\ge1}
       (\overline{Q_n}\cup\overline{P_n})\right).
\end{equation}
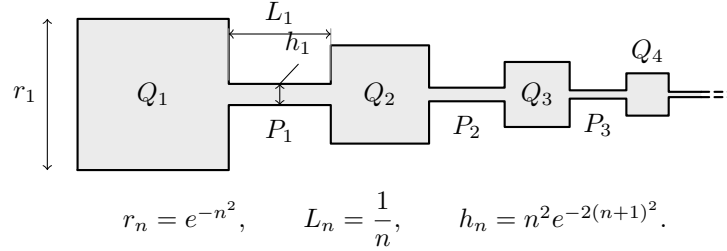
\begin{figure}[htbp]
\centering
\begin{tikzpicture}[x=1cm,y=1cm,font=\small,line cap=round,line join=round]
\path[fill=black!8]
(0,-1)--(0,1)--(2,1)--(2,.14)--(3.35,.14)
--(3.35,.65)--(4.65,.65)--(4.65,.09)--(5.65,.09)
--(5.65,.43)--(6.51,.43)--(6.51,.055)--(7.26,.055)
--(7.26,.28)--(7.82,.28)--(7.82,.035)--(8.25,.035)
--(8.25,-.035)--(7.82,-.035)--(7.82,-.28)--(7.26,-.28)
--(7.26,-.055)--(6.51,-.055)--(6.51,-.43)--(5.65,-.43)
--(5.65,-.09)--(4.65,-.09)--(4.65,-.65)--(3.35,-.65)
--(3.35,-.14)--(2,-.14)--(2,-1)--(0,-1)--cycle;
\draw[thick]
(0,-1)--(0,1)--(2,1)--(2,.14)--(3.35,.14)
--(3.35,.65)--(4.65,.65)--(4.65,.09)--(5.65,.09)
--(5.65,.43)--(6.51,.43)--(6.51,.055)--(7.26,.055)
--(7.26,.28)--(7.82,.28)--(7.82,.035)--(8.25,.035)
(8.25,-.035)--(7.82,-.035)--(7.82,-.28)--(7.26,-.28)
--(7.26,-.055)--(6.51,-.055)--(6.51,-.43)--(5.65,-.43)
--(5.65,-.09)--(4.65,-.09)--(4.65,-.65)--(3.35,-.65)
--(3.35,-.14)--(2,-.14)--(2,-1)--(0,-1);
\draw[thick,densely dashed] (8.25,.035)--(8.65,.035);
\draw[thick,densely dashed] (8.25,-.035)--(8.65,-.035);
\node at (1,0) {$Q_1$};
\node at (4,0) {$Q_2$};
\node at (6.08,0) {$Q_3$};
\node[above] at (7.54,.28) {$Q_4$};
\node[below] at (2.675,-.2) {$P_1$};
\node[below] at (5.15,-.15) {$P_2$};
\node[below] at (6.885,-.12) {$P_3$};
\draw[<->] (-.4,-1)--node[left] {$r_1$}(-.4,1);
\draw[thin,black!40] (2,.19)--(2,.87) (3.35,.19)--(3.35,.87);
\draw[<->] (2,.8)--node[above] {$L_1$}(3.35,.8);
\draw[<->] (2.675,-.14)--(2.675,.14);
\draw[thin] (2.675,.14)--(2.92,.4) node[above] {$h_1$};
\node at (4.2,-1.65) {$r_n=e^{-n^2},\qquad L_n=\dfrac1n,\qquad h_n=n^2e^{-2(n+1)^2}.$};

\end{tikzpicture}
\caption{Square rooms $Q_n$ of side $r_n$, joined by rectangular passages $P_n$ of length $L_n$ and width $h_n$. The interfaces are open inside the domain. The drawing is schematic. Room sizes and passage widths are enlarged for visibility. Dashed lines indicate continuation to infinity.}
\label{fig:rooms-passages}
\end{figure}

Figure~\ref{fig:rooms-passages} shows the arrangement schematically. The open interfaces between adjacent pieces belong to $\Om$.
Since $h_n<r_{n+1}<r_n$, each passage is narrower than both adjacent rooms.
Every vertical section is an interval containing the horizontal axis. Vertical contraction onto that axis shows that $\Om$ is contractible.
Moreover, $\sum_nL_n=\infty$, whereas
\begin{equation}\label{eq:rp-area}
 |\Om|=\sum_{n\ge1}(r_n^2+h_nL_n)
       =\sum_{n\ge1}\left(e^{-2n^2}+n e^{-2(n+1)^2}\right)<\infty.
\end{equation}
Every bounded horizontal interval meets only finitely many pieces.
The boundary is locally Lipschitz at every finite boundary point, but
the shrinking widths prevent a uniform positive scale for boundary charts.
Write $\Om_{\ge M}=\Om\cap\{(x,y):x>x_M\}$. Up to sets of area zero,
this is the union of $Q_n$ and $P_n$ for $n\ge M$.

\begin{theorem}\label{thm:rooms-passages}
The domain \eqref{eq:rp-domain} is an unbounded simply connected planar
domain of finite area with locally Lipschitz boundary. The embedding
$H^1(\Om)\hookrightarrow L^2(\Om)$ is compact. More precisely,
there exists $C<\infty$ such that
\begin{equation}\label{eq:rp-tail}
 \int_{\Om_{\ge M}}|u|^2
 \le\frac{C}{M^2}\left(\int_\Om|u|^2+\int_\Om|\nabla u|^2\right),
 \qquad M\ge2,\quad u\in H^1(\Om).
\end{equation}
Its Neumann eigenvalues, counted with multiplicity, satisfy
\[
 0=\lambda_0<\lambda_1<\frac45<\frac{10}{9}\le\lambda_2.
\]
In particular, the first positive eigenvalue is simple. Under
Hypothesis~\ref{hyp:smooth}, every real eigenfunction associated with
$\lambda_1$ satisfies the boundary-layer identities \eqref{eq:HS}.
\end{theorem}

\begin{proof}
\emph{Compactness.}
First record explicit local estimates that will also be used for simplicity.
For a square $Q$ of side $r$, an opening $I$ of length $h\le r$ in one
side, and the respective averages $a=|Q|^{-1}\int_Q u$ and
$b=h^{-1}\int_Iu$, one has
\begin{equation}\label{eq:rp-square-opening}
 \int_Q|u-a|^2\le\frac{r^2}{\pi^2}\int_Q|\nabla u|^2,
 \qquad
 |b-a|^2\le\frac{r}{2h}\int_Q|\nabla u|^2.
\end{equation}
The first estimate is the Poincar\'e inequality on a square. To verify
the second with its stated constant, take $Q=(0,r)^2$ and the relevant
side to be $J=\{0\}\times(0,r)$. For
$g(y)=r^{-1}\int_0^r u(x,y)\,dx$, the one-dimensional Poincar\'e
inequality gives
\[
 \|g-a\|_{L^2(0,r)}^2\le\frac r{\pi^2}\int_Q|\partial_yu|^2.
\]
Also,
\[
 u(0,y)-g(y)=-\frac1r\int_0^r(r-x)\partial_xu(x,y)\,dx,
 \qquad
 \|u(0,\cdot)-g\|_{L^2(0,r)}^2
 \le\frac r3\int_Q|\partial_xu|^2.
\]
The triangle inequality followed by Cauchy--Schwarz yields
\[
 \int_J|u-a|^2
 \le r\left(\frac13+\frac1{\pi^2}\right)\int_Q|\nabla u|^2
 \le\frac r2\int_Q|\nabla u|^2.
\]
Averaging over $I\subset J$ proves the second estimate in
\eqref{eq:rp-square-opening}. Density and continuity of the trace
justify the calculation for $H^1(Q)$ functions.

Let $a_n$ be the average of $u$ on $Q_n$. The matching traces at the
openings and the elementary passage-end estimate
\[
 |b_{\rm right}-b_{\rm left}|^2
 \le\frac{L_n}{h_n}\int_{P_n}|\nabla u|^2
\]
imply
\begin{equation}\label{eq:rp-increments}
 |a_{n+1}-a_n|^2\le\rho_n E_n,
 \qquad
 \rho_n=\frac{r_n/2+L_n+r_{n+1}/2}{h_n},
\end{equation}
where
\[
 E_n=\int_{Q_n}|\nabla u|^2+\int_{P_n}|\nabla u|^2
                  +\int_{Q_{n+1}}|\nabla u|^2.
\]
Indeed, each of the three comparisons is bounded by the square root
of its resistance coefficient times the energy in the corresponding
piece; Cauchy--Schwarz gives \eqref{eq:rp-increments}.

The mass in a passage satisfies
\begin{align}
 \int_{P_n}|u|^2\le {}&4h_nL_n|a_n|^2
       +2L_nr_n\int_{Q_n}|\nabla u|^2
       +\frac{8L_n^2}{\pi^2}\int_{P_n}|\nabla u|^2.
 \label{eq:rp-passage-mass}
\end{align}
For completeness, let $m(s)$ be the transverse average on $P_n$,
$0<s<L_n$, and $b=m(0)$. The transverse Poincar\'e inequality and
the one-dimensional Dirichlet--Neumann inequality give
\begin{align*}
 \int_{P_n}|u-m(s)|^2&\le\frac{h_n^2}{\pi^2}
                                      \int_{P_n}|\partial_yu|^2,\\
 h_n\int_0^{L_n}|m(s)|^2\,ds
 &\le2h_nL_n|b|^2+\frac{8L_n^2}{\pi^2}
                                      \int_{P_n}|\partial_xu|^2.
\end{align*}
Use $|b|^2\le2|a_n|^2+(r_n/h_n)\int_{Q_n}|\nabla u|^2$,
the orthogonal decomposition into transverse average and fluctuation,
and $h_n\le L_n$ to obtain \eqref{eq:rp-passage-mass}.
For the rooms,
\begin{equation}\label{eq:rp-room-mass}
 \int_{Q_n}|u|^2
 \le r_n^2|a_n|^2+\frac{r_n^2}{\pi^2}\int_{Q_n}|\nabla u|^2.
\end{equation}

Put $E=\int_\Om|\nabla u|^2$, $A_1=0$, and
$A_n=\sum_{k=1}^{n-1}\rho_k$ for $n\ge2$.
Since $\sum_kE_k\le2E$, telescoping \eqref{eq:rp-increments} gives
\begin{equation}\label{eq:rp-average-global}
 |a_n|^2\le2|a_1|^2+4A_nE.
\end{equation}
Set $T_n=e^{2(n+1)^2}/n^3$. Because $r_n,r_{n+1}\le1/n$,
one has $T_n\le\rho_n\le2T_n$, and for $n\ge2$,
\begin{equation}\label{eq:rp-ratio}
 \frac{T_{n-1}}{T_n}
 =\left(\frac n{n-1}\right)^3e^{-4n-2}
 \le8e^{-10}<\frac1{100}.
\end{equation}
Consequently,
\[
 A_n\le C\frac{e^{2n^2}}{(n-1)^3}\quad(n\ge2),\qquad
 W_n:=r_n^2+4h_nL_n
 =e^{-2n^2}(1+4n e^{-4n-2})\le\frac{11}{10}e^{-2n^2}.
\]
It follows that
\[
 \sum_{n\ge M}W_nA_n\le\frac C{M^2},
 \qquad \sum_{n\ge M}W_n\le\frac C{M^2},\qquad M\ge2.
\]
The fluctuation coefficients in \eqref{eq:rp-passage-mass} and
\eqref{eq:rp-room-mass} are bounded by $C/n^2$.
Summing those estimates and using \eqref{eq:rp-average-global} gives
\[
 \int_{\Om_{\ge M}}|u|^2
 \le\frac C{M^2}(|a_1|^2+E).
\]
Since $|a_1|^2\le |Q_1|^{-1}\int_{Q_1}|u|^2$, this proves
\eqref{eq:rp-tail}. Each initial segment
$U_M=\Om\cap\{(x,y):x<x_M\}$ is a bounded Lipschitz polygonal
domain. Rellich compactness on these initial segments, a diagonal
subsequence, and the uniform tail estimate prove compactness on $\Om$.
The Neumann resolvent is therefore compact. Connectedness and finite
area imply that its zero eigenspace consists precisely of the constants.

\medskip
\emph{Simplicity.}
For an upper bound on $\lambda_1$, let $f$ vanish on $Q_1$, increase
linearly from zero to one along $P_1$, and equal one on the remaining
domain. Then $f\in H^1(\Om)$ and
$\int_\Om|\nabla f|^2=h_1/L_1=e^{-8}$.
Write $m_2=e^{-8}=|Q_2|=|P_1|$ and $m_1=e^{-2}=|Q_1|$.
Minimizing over constants, first on the whole domain and then on the
first three pieces, gives
\begin{align*}
 \int_\Om|f-f_\Om|^2
 &=\min_{c\in\R}\int_\Om|f-c|^2\\
 &\ge\min_{c\in\R}\int_{Q_1\cup P_1\cup Q_2}|f-c|^2
 =\frac{4m_2}{3}-\frac{9m_2^2}{4(m_1+2m_2)}.
\end{align*}
Here $f_\Om$ denotes the average on $\Om$.
Since $m_1/m_2=e^6>25$, the last expression exceeds $5m_2/4$.
The Rayleigh principle yields $\lambda_1<4/5$.

To bound $\lambda_2$, consider the subspace
\[
 V=\left\{u\in H^1(\Om):\int_{Q_1}u=0,\quad\int_{Q_2}u=0\right\}.
\]
It suffices to prove
\begin{equation}\label{eq:rp-codim-two}
 \int_\Om|u|^2\le\frac9{10}\int_\Om|\nabla u|^2,
 \qquad u\in V.
\end{equation}
Indeed, these are two independent $L^2$-orthogonality conditions,
so the min--max principle gives $\lambda_2\ge10/9$.
Fix $u\in V$, so that $a_1=a_2=0$. Estimates
\eqref{eq:rp-passage-mass}--\eqref{eq:rp-room-mass} give
\begin{align}
 \int_{Q_1\cup P_1}|u|^2
 &\le\left(\frac{e^{-2}}{\pi^2}+2e^{-1}\right)
                      \int_{Q_1}|\nabla u|^2
       +\frac8{\pi^2}\int_{P_1}|\nabla u|^2\notag\\
 &\le\frac9{10}\int_{Q_1\cup P_1}|\nabla u|^2.
 \label{eq:rp-first-block}
\end{align}

For the tail, put $E_T=\int_{\Om_{\ge2}}|\nabla u|^2$ and
$S_n=\sum_{k=2}^{n-1}\rho_k$ for $n\ge3$.
Since $a_2=0$, \eqref{eq:rp-increments} gives
\[
 |a_n|^2\le2S_nE_T.
\]
For $k\ge2$,
\[
 \rho_k=\left(1+\frac{k(r_k+r_{k+1})}{2}\right)T_k
 \le\frac{21}{20}T_k,
\]
because $k(r_k+r_{k+1})/2\le k e^{-k^2}\le2e^{-4}<1/20$.
Together with \eqref{eq:rp-ratio}, this gives
\[
 S_n\le\frac{21}{20}\frac{100}{99}T_{n-1}
 <\frac{11}{10}\frac{e^{2n^2}}{(n-1)^3}.
\]
Using the bound on $W_n$ from the compactness argument,
\begin{equation}\label{eq:rp-tail-series}
 \sum_{n=3}^\infty W_nS_n
 \le\frac{121}{100}\sum_{k=2}^\infty k^{-3}
 \le\frac{121}{400},
\end{equation}
where $\sum_{k=2}^\infty k^{-3}\le1/8+\int_2^\infty t^{-3}\,dt=1/4$.
For $n\ge2$ the fluctuation coefficients satisfy
\[
 \frac{r_n^2}{\pi^2}+2L_nr_n\le\frac14,
 \qquad \frac{8L_n^2}{\pi^2}\le\frac14.
\]
Thus \eqref{eq:rp-passage-mass}--\eqref{eq:rp-room-mass}, $a_2=0$,
and \eqref{eq:rp-tail-series} imply
\[
 \int_{\Om_{\ge2}}|u|^2
 \le\sum_{n=3}^\infty W_n|a_n|^2+\frac14E_T
 \le\left(\frac{121}{200}+\frac14\right)E_T
 =\frac{171}{200}E_T\le\frac9{10}E_T.
\]
Combining this with \eqref{eq:rp-first-block} proves
\eqref{eq:rp-codim-two}, and hence simplicity. The conditional hot
spots conclusion follows from Theorem~\ref{thm:main}.
\end{proof}

As an unbounded domain of finite area, this example also admits no
polynomial Neumann heat-kernel bound, by
Remark~\ref{rem:hornNoPolynomial}. The numerical bounds above are
deliberately nonoptimal. Their purpose is to separate the first two
positive eigenvalues without identifying the nodal set.

\section{On removing the simplicity assumption}
\label{sec:withoutSimplicity}

The examples given above do not exhaust the domains
to which Theorem~\ref{thm:main} or Corollary~\ref{cor:HKhot}
applies. Several results in the literature establish simplicity of the
first positive Neumann eigenvalue for specific classes of domains,
often in connection with results on the hot spots conjecture.
These classes include sufficiently elongated bounded convex
planar domains \cite[Proposition~2.4]{BB},
nonsquare lip domains \cite[Theorem~1(i)]{AB}, the class of smooth
convex planar domains covered by \cite[Theorem~2.6]{ABMirror},
and all non-equilateral triangles \cite[Theorem~1.2]{Siudeja}.
Moreover, among simply connected planar polygons with any fixed
number of vertices greater than or equal to four, almost every
polygon has simple Neumann spectrum. See
\cite[Theorem~1.1]{HJPolygons} for the precise statement and the
meaning of ``almost every.'' Convexity is not required. The latter paper establishes generic simplicity of the spectrum,
but does not prove a hot spots theorem.

These bounded Lipschitz domains have compact embedding
$H^1\hookrightarrow L^2$ and satisfy polynomial Neumann heat-kernel
bounds. Consequently, under Hypothesis~\ref{hyp:smooth},
Theorem~\ref{thm:main} and Corollary~\ref{cor:HKhot} apply to them.
For lip domains and triangles, the hot spots conclusion is already
known independently. See \cite[Theorem~1(ii)]{AB} and
\cite[erratum, Theorems~A and~4.1]{JudgeMondal}, respectively.

The explicit spectral comparisons in these notes concern domains
with nonsmooth boundary, particularly non-Lipschitz cusps, and
the unbounded Gaussian horn of finite area. The exponential cusp
and the Gaussian horn satisfy the compactness and simplicity
assumptions but admit no polynomial Neumann heat-kernel bound,
illustrating the broader scope of the compactness approach.

Simplicity is not necessary for a domain to satisfy the hot spots
conclusion. The square and the disk are bounded simply connected
Lipschitz domains whose first positive Neumann eigenvalue has
multiplicity two, yet every eigenfunction in the corresponding
eigenspace satisfies the nonstrict hot spots conclusion. These
examples motivate the question of whether the simplicity assumption
can be removed more generally.

The preceding results suggest two questions. Throughout, assume
Hypothesis~\ref{hyp:smooth}, and impose no restriction on the
multiplicity of the first positive Neumann eigenvalue.

\paragraph{Question 1: Lipschitz domains.}
Let $\Om\subset\R^2$ be a bounded simply connected Lipschitz domain.
Does every nonzero real eigenfunction $u$ associated with
$\lambda_1(\Om)$ satisfy
\[
 \sup_\Om u
 =
 \lim_{\delta\downarrow0}\sup_{B_\delta(\Om)}u,
 \qquad
 \inf_\Om u
 =
 \lim_{\delta\downarrow0}\inf_{B_\delta(\Om)}u,
\]
where
\[
 B_\delta(\Om)
 =
 \{x\in\Om:\dist(x,\partial\Om)<\delta\}?
\]
In particular, does this hold for every bounded simply connected polygonal domain,
including those whose first positive Neumann eigenvalue
may be multiple?

The result of Deng, Jiang, and Yang \cite{DengJiangYang2026}
does not settle this question, since the polygons considered
here need not be convex.

\paragraph{Question 2: Domains with compact embedding.}
Does the same conclusion hold for every simply connected planar
domain $\Om$ of finite area for which
$H^1(\Om)\hookrightarrow L^2(\Om)$ is compact, without assuming
boundedness or any regularity of $\partial\Om$?

In both questions, eigenfunctions are understood through their
smooth interior representatives, and infinite supremal or infimal
values are allowed. Question~2 includes Question~1, since bounded
Lipschitz domains have compact Sobolev embedding. The distinction
separates the multiplicity question for bounded Lipschitz domains
from its extension to domains with non-Lipschitz boundaries or
unbounded domains of finite area.

The approximation argument gives a partial answer. The
boundary-layer identities hold for at least one nonzero real
first positive Neumann eigenfunction. The following proposition
therefore gives a boundary-layer version of (HS3) in \cite{BB}.
For bounded domains and eigenfunctions continuous on the closure,
these identities are equivalent to the boundary inequalities
in (HS3). The two questions above ask for the corresponding
version of (HS2).

\begin{proposition}\label{prop:oneEigenfunction}
Assume Hypothesis~\ref{hyp:smooth}. Every simply connected planar domain $\Om$ of finite area with compact embedding $H^1(\Om)\hookrightarrow L^2(\Om)$ admits a nonzero real first positive Neumann eigenfunction satisfying \eqref{eq:HS}, regardless of the multiplicity of $\lambda_1(\Om)$.
\end{proposition}

\begin{proof}
Let $F:\D\to\Om$ be a conformal bijection, set $\Om_r=F(r\D)$, and choose normalized real first positive Neumann eigenfunctions $u_j$ on $\Om_{r_j}$, where $r_j\uparrow1$. Proposition~\ref{prop:spectral}, whose compactness and variational arguments do not require simplicity, gives a subsequence converging locally uniformly in $\Om$ to a normalized first positive Neumann eigenfunction $v$. If the supremum identity in \eqref{eq:HS} failed for $v$, Lemma~\ref{lem:superlevel} would give, for large $j$, an interior local maximum of $u_j$ on the bounded smooth simply connected domain $\Om_{r_j}$, contradicting Hypothesis~\ref{hyp:smooth}. The infimum identity follows by applying the same argument to $-v$.
\end{proof}

The remaining difficulty is to approximate each prescribed eigenfunction in the first positive eigenspace by first positive eigenfunctions on smooth approximating domains. When the limiting eigenvalue is multiple, the associated spectral cluster may split. Convergence of the entire cluster does not resolve this difficulty. Its elements need not be first positive eigenfunctions on the approximants, and Hypothesis~\ref{hyp:smooth} does not apply to their arbitrary linear combinations. The present argument therefore proves the conclusion for at least one limiting eigenfunction, but not for every eigenfunction in the eigenspace.

\begin{remark}
The same obstruction arises in the heat-kernel approach through spectral stability. The convergence of the spectral subspace associated with a multiple limiting eigenvalue does not ensure that each prescribed limiting eigenfunction can be approximated by first positive eigenfunctions. Thus this alternative approach does not by itself remove the simplicity assumption.
\end{remark}

\appendix
\section{Further classes with polynomial heat-kernel bounds}\label{app:classes}
This appendix presents two additional criteria for polynomial Neumann heat-kernel bounds and computes the quasihyperbolic order of the power cusps. These results supplement the examples and are not needed for the main proofs.

\subsection{Quasihyperbolic H\"older domains}
Here ``H\"older domain'' has the quasihyperbolic meaning used in \cite{BanuelosIU} and the references therein, rather than the boundary-graph meaning of Section~\ref{sec:examples}. Write $\delta(z)=\dist(z,\partial\Om)$ and define
\[
 k_\Om(z,w)=\inf_\gamma\int_\gamma\frac{|d\zeta|}{\delta(\zeta)},
\]
where the infimum runs over rectifiable curves $\gamma$ joining $z$ to $w$ in $\Om$.

\begin{remark}[Comparison with the hyperbolic metric]
For a proper simply connected planar domain $\Om$, normalize the hyperbolic metric by $\rho_{\D}(z)=2/(1-|z|^2)$, write $\rho_\Om$ for its density on $\Om$, and denote its distance by $h_\Om$. The Koebe quarter theorem and Schwarz's lemma give
\begin{equation}\label{eq:hyperbolicComparison}
 \frac1{2\delta(z)}\le\rho_\Om(z)\le\frac2{\delta(z)},\qquad
 \frac12k_\Om(z,w)\le h_\Om(z,w)\le2k_\Om(z,w).
\end{equation}
Thus either distance gives the same H\"older growth conditions below, up to constants. The hyperbolic distance is conformally invariant, whereas the quasihyperbolic distance generally is not.
\end{remark}

Following \cite{BanuelosIU} for order zero and the quasihyperbolic conventions in \cite[Section 1.2]{Aikawa}, a bounded domain is a \emph{H\"older domain of order zero} if, for some $z_0\in\Om$ and $A,B>0$,
\begin{equation}\label{eq:QHBzero}
 k_\Om(z,z_0)\le A\log\frac{\delta(z_0)}{\delta(z)}+B
 \qquad(z\in\Om).
\end{equation}
For $\eta>0$, it is a \emph{H\"older domain of order $\eta$} if, for some $z_0\in\Om$ and $A,B>0$,
\begin{equation}\label{eq:QHBpositive}
 k_\Om(z,z_0)\le
 A\left(\frac{\delta(z_0)}{\delta(z)}\right)^\eta+B
 \qquad(z\in\Om).
\end{equation}
The positive-order condition is the condition QHB$(\alpha)$ of \cite[Section 1.2]{Aikawa}, with Aikawa's $\alpha$ renamed $\eta$ here to avoid a clash with the cusp exponent. The order measures the growth of the quasihyperbolic distance, not the regularity of a boundary graph. Even the order-zero class contains bounded John domains, hence uniform domains and quasidisks such as the Koch snowflake. The next proposition gives the heat-kernel bound for quasihyperbolic H\"older domains of order zero. The power cusps of Section~\ref{subsec:powerCusps} provide examples of positive order.

\begin{proposition}\label{prop:QHBzeroHK}
Every bounded planar domain satisfying \eqref{eq:QHBzero} satisfies \eqref{eq:examplesSob} for every finite $q>2$. Consequently, for every $N>2$ there is a constant $C_{\Om,N}$ such that
\[
 p_\Om(t,z,w)\le C_{\Om,N}t^{-N/2},\qquad 0<t\le1,
\]
for almost every $z,w\in\Om$. In particular, $H^1(\Om)\hookrightarrow L^2(\Om)$ is compact.
\end{proposition}
\begin{proof}
In \eqref{eq:QHBzero} we may enlarge $A$ and adjust $B$: the function $\log(\delta(z_0)/\delta(z))$ is bounded below on a bounded domain, so a larger coefficient is compensated by a larger additive constant. Write the coefficient as $1/\beta$ with $0<\beta\le1$. The theorem of Koskela--Onninen--Tyson \cite{KOT} gives the Sobolev--Poincar\'e inequality
\[
 \|u-u_\Om\|_{L^q(\Om)}\le C\|\nabla u\|_{L^p(\Om)},
 \qquad u_\Om=|\Om|^{-1}\int_\Om u,
\]
provided $\max\{1,2-2\beta\}<p<2$ and $p\le q<2\beta p/(2-p)$. Given any finite $q>2$, choose $p<2$ sufficiently close to two. Finiteness of the area gives
\[
 \|\nabla u\|_p\le |\Om|^{1/p-1/2}\|\nabla u\|_2,
 \qquad
 \|u_\Om\|_q\le |\Om|^{1/q-1/2}\|u\|_2.
\]
Together these give \eqref{eq:examplesSob}. Proposition 3 of \cite{BD} then yields the kernel bound with $N=2q/(q-2)$, and hence with every $N>2$. Compactness of the embedding follows from Remark~\ref{rem:adds}.
\end{proof}
The argument does not yield the endpoint bound $Ct^{-1}$.

\subsection{Quasihyperbolic order of the power cusps}
The power cusps $\Om_\alpha$ of Section~\ref{subsec:powerCusps} have positive quasihyperbolic order. To avoid confusing the cusp exponent $\alpha>1$ with the order, put $\eta=1-1/\alpha\in(0,1)$ and fix $z_0=(1/2,0)$. Along the axis, as $x\downarrow0$,
\[
 \delta(x,0)\asymp x^\alpha,\qquad
 k_{\Om_\alpha}((x,0),z_0)\asymp x^{1-\alpha}
 \asymp\delta(x,0)^{-\eta}.
\]
For the upper bound integrate $r^{-\alpha}$ along the axis. For the lower bound, any path from $(x,0)$ to $z_0$ must traverse the horizontal coordinates between $x$ and $1/2$, while its boundary distance at horizontal coordinate $r$ is at most $r^\alpha$. Near the tip, the bounded slopes of the graphs make boundary distance comparable to the vertical gap from the nearer graph. Thus, for an off-axis point $(x,y)$ near the tip, a vertical segment to the axis costs at most $C[1+\log(x^\alpha/\delta(x,y))]$ in quasihyperbolic length. Since $\delta(x,y)\le x^\alpha$ and $\log(1/d)\le C_\eta d^{-\eta}$ for small $d$, this gives \eqref{eq:QHBpositive} with $\eta=1-1/\alpha$. The remaining Lipschitz boundary contributes only logarithmic growth. The axis estimate excludes order zero and every positive order smaller than $\eta$.

Thus, for every $0<\eta<1$ there is a positive-order example with a polynomial Neumann bound, whose exponent is
\[
 N=\alpha+1=\frac{2-\eta}{1-\eta}.
\]
This exponent is optimal for this cusp family \cite[Example 6]{BD}. Whether the positive-order quasihyperbolic condition alone implies a polynomial Neumann bound on every domain satisfying it is not addressed here.

\subsection{Finite decompositions and a slit domain}
\begin{proposition}\label{prop:finiteCover}
Let $\Om$ be bounded. Suppose there are finitely many open subdomains $\Om_1,\ldots,\Om_m\subset\Om$ whose union covers $\Om$ up to a null set. If each $\Om_j$ satisfies \eqref{eq:examplesSob} with an exponent $q_j>2$, then $\Om$ satisfies \eqref{eq:examplesSob} with $q=\min_jq_j>2$. Consequently $\Om$ has a polynomial Neumann heat-kernel bound and compact Sobolev embedding.
\end{proposition}
\begin{proof}
By H\"older's inequality and $|\Om_j|<\infty$, the estimate on each piece holds with the common exponent $q$. Restriction takes $u\in H^1(\Om)$ into $H^1(\Om_j)$. Since $2/q<1$,
\begin{align*}
 \norm{u}{L^q(\Om)}^2
 &\le\left(\sum_{j=1}^m\norm{u}{L^q(\Om_j)}^q\right)^{2/q}
 \le\sum_{j=1}^m\norm{u}{L^q(\Om_j)}^2\\
 &\le C\sum_{j=1}^m\norm{u}{H^1(\Om_j)}^2
 \le Cm\norm{u}{H^1(\Om)}^2.
\end{align*}
The equivalence between \eqref{eq:examplesSob} and the heat-kernel bound, together with Remark~\ref{rem:adds}, completes the proof. No extension or matching of functions across the interfaces is needed.
\end{proof}

For example, the slit disk
\[
 \Om=\D\setminus\{(x,0):0\le x<1\}
\]
is bounded and simply connected, but it is neither a Lipschitz domain nor a Jordan domain, and it does not satisfy the one-sided H\"older graph condition of Section~\ref{sec:examples} along the slit. Its upper and lower half-disks are Lipschitz subdomains covering it up to a null set, so Proposition~\ref{prop:finiteCover} supplies a polynomial bound with any chosen $N>2$. More generally, finite decompositions into H\"older graph domains or quasidisks give the same conclusion for some finite $N$.


\begin{thebibliography}{99}
\phantomsection
\addcontentsline{toc}{section}{References}

\bibitem{Adams1973}
R.~A.~Adams,
\emph{Some integral inequalities with applications to the imbedding of Sobolev spaces defined over irregular domains},
Trans.\ Amer.\ Math.\ Soc.\ \textbf{178} (1973), 401--429.
\href{https://doi.org/10.1090/S0002-9947-1973-0322494-0}{doi:10.1090/S0002-9947-1973-0322494-0}.

\bibitem{Ahlfors}
L.~V.~Ahlfors,
\emph{Conformal Invariants: Topics in Geometric Function Theory},
McGraw-Hill, New York, 1973.

\bibitem{Aikawa}
H.~Aikawa,
\emph{Intrinsic ultracontractivity via capacitary width},
Rev.\ Mat.\ Iberoam.\ \textbf{31} (2015), 1041--1106.
\href{https://doi.org/10.4171/RMI/863}{doi:10.4171/RMI/863}.

\bibitem{AB}
R.~Atar and K.~Burdzy,
\emph{On Neumann eigenfunctions in lip domains},
J.\ Amer.\ Math.\ Soc.\ \textbf{17} (2004), 243--265.
\href{https://doi.org/10.1090/S0894-0347-04-00453-9}{doi:10.1090/S0894-0347-04-00453-9}.

\bibitem{ABMirror}
R.~Atar and K.~Burdzy,
\emph{Mirror couplings and Neumann eigenfunctions},
Indiana Univ.\ Math.\ J.\ \textbf{57} (2008), 1317--1351.
\href{https://doi.org/10.1512/iumj.2008.57.3222}{doi:10.1512/iumj.2008.57.3222};
\url{https://arxiv.org/abs/math/0607400}.

\bibitem{AtarBurdzy2026}
R.~Atar and K.~Burdzy,
\emph{A starlike 3-dimensional counterexample to the hot spots conjecture},
arXiv:2610.19908 (2026).
\url{https://arxiv.org/abs/2610.19908}.

\bibitem{BanuelosIU}
R.~Ba\~{n}uelos,
\emph{Intrinsic ultracontractivity and eigenfunction estimates for Schr\"odinger operators},
J.\ Funct.\ Anal.\ \textbf{100} (1991), 181--206.
\href{https://doi.org/10.1016/0022-1236(91)90107-G}{doi:10.1016/0022-1236(91)90107-G}.

\bibitem{BB}
R.~Ba\~{n}uelos and K.~Burdzy,
\emph{On the ``hot spots'' conjecture of J.~Rauch},
J.\ Funct.\ Anal.\ \textbf{164} (1999), 1--33.

\bibitem{BCDuke}
R.~Ba\~{n}uelos and T.~Carroll,
\emph{Brownian motion and the fundamental frequency of a drum},
Duke Math.\ J.\ \textbf{75} (1994), 575--602;
Addendum, Duke Math.\ J.\ \textbf{82} (1996), 227.
\href{https://doi.org/10.1215/S0012-7094-96-08210-1}{doi:10.1215/S0012-7094-96-08210-1} (addendum).

\bibitem{BCAIF}
R.~Ba\~{n}uelos and T.~Carroll,
\emph{Extremal problems for conditioned Brownian motion and the hyperbolic metric},
Ann.\ Inst.\ Fourier \textbf{50} (2000), 1507--1532.
\href{https://doi.org/10.5802/aif.1798}{doi:10.5802/aif.1798}.

\bibitem{BanuelosDavis1989}
R.~Ba\~nuelos and B.~Davis,
\emph{Heat kernel, eigenfunctions, and conditioned Brownian motion in planar domains},
J.\ Funct.\ Anal.\ \textbf{84} (1989), 188--200.
\href{https://doi.org/10.1016/0022-1236(89)90118-3}{doi:10.1016/0022-1236(89)90118-3}.

\bibitem{BP}
R.~Ba\~{n}uelos and M.~M.~H.~Pang,
\emph{Stability and approximations of eigenvalues and eigenfunctions for the Neumann Laplacian, part I},
Electron.\ J.\ Differential Equations \textbf{2008} (2008), no.~145, 1--13.
\url{https://ejde.math.txstate.edu/Volumes/2008/145/abstr.html}.

\bibitem{BVan}
R.~Ba\~{n}uelos and M.~van den Berg,
\emph{Dirichlet eigenfunctions for horn-shaped regions and Laplacians on cross sections},
J.\ London Math.\ Soc.\ (2) \textbf{53} (1996), 503--511.

\bibitem{BEW}
B.~M.~Brown, W.~D.~Evans, and I.~G.~Wood,
\emph{Some spectral properties of Rooms and Passages domains and their skeletons},
preprint (2013), \url{https://arxiv.org/abs/1301.4396}.

\bibitem{BurdzyNote}
K.~Burdzy,
\emph{Removing boundary smoothness from the hot-spots conclusion in the simple-eigenvalue case},
research note, arXiv:2610.12622 (October 2026).
\url{https://arxiv.org/abs/2610.12622}.

\bibitem{BC}
K.~Burdzy and Z.-Q.~Chen,
\emph{Weak convergence of reflecting Brownian motions},
Electron.\ Comm.\ Probab.\ \textbf{3} (1998), paper no.~4, 29--33.
\href{https://doi.org/10.1214/ECP.v3-990}{doi:10.1214/ECP.v3-990}.

\bibitem{Burenkov}
V.~I.~Burenkov,
\emph{Sobolev Spaces on Domains},
Teubner-Texte zur Mathematik, vol.~137, Teubner, Stuttgart, 1998.

\bibitem{BD}
V.~I.~Burenkov and E.~B.~Davies,
\emph{Spectral stability of the Neumann Laplacian},
J.\ Differential Equations \textbf{186} (2002), 485--508.
\href{https://doi.org/10.1016/S0022-0396(02)00033-5}{doi:10.1016/S0022-0396(02)00033-5};
arXiv:math/0106160.

\bibitem{CM}
M.~Cranston and T.~R.~McConnell,
\emph{The lifetime of conditioned Brownian motion},
Z.\ Wahrsch.\ Verw.\ Gebiete \textbf{65} (1983), 1--11.
\href{https://doi.org/10.1007/BF00534989}{doi:10.1007/BF00534989}.

\bibitem{Davies}
E.~B.~Davies,
\emph{Heat Kernels and Spectral Theory},
Cambridge Tracts in Mathematics, vol.~92,
Cambridge University Press, Cambridge, 1989.

\bibitem{DS}
E.~B.~Davies and B.~Simon,
\emph{Ultracontractivity and the heat kernel for Schr\"odinger operators and Dirichlet Laplacians},
J.\ Funct.\ Anal.\ \textbf{59} (1984), 335--395.
\href{https://doi.org/10.1016/0022-1236(84)90076-4}{doi:10.1016/0022-1236(84)90076-4}.

\bibitem{DengJiangYang2026}
H.~Deng, X.~Jiang, and X.~Yang,
\emph{Near-boundary rigidity and strict maximum modulus for Neumann eigenfunctions on convex planar domains},
arXiv:2610.18546 (2026).
\url{https://arxiv.org/abs/2610.18546}.

\bibitem{GM}
J.~B.~Garnett and D.~E.~Marshall,
\emph{Harmonic Measure},
New Mathematical Monographs, vol.~2,
Cambridge University Press, Cambridge, 2005.

\bibitem{GSC}
P.~Gyrya and L.~Saloff-Coste,
\emph{Neumann and Dirichlet heat kernels in inner uniform domains},
Ast\'erisque \textbf{336}, Soci\'et\'e Math\'ematique de France, Paris, 2011.
\url{https://numdam.org/issues/AST_2011__336__R1_0/}.

\bibitem{HaleRaugel}
J.~K.~Hale and G.~Raugel,
\emph{Reaction-diffusion equation on thin domains},
J.\ Math.\ Pures Appl.\ (9) \textbf{71} (1992), 33--95.

\bibitem{Hayman}
W.~K.~Hayman,
\emph{Some bounds for principal frequency},
Applicable Anal.\ \textbf{7} (1977/78), 247--254.

\bibitem{HSS}
R.~Hempel, L.~A.~Seco, and B.~Simon,
\emph{The essential spectrum of Neumann Laplacians on some bounded singular domains},
J.\ Funct.\ Anal.\ \textbf{102} (1991), 448--483.
\href{https://doi.org/10.1016/0022-1236(91)90130-W}{doi:10.1016/0022-1236(91)90130-W}.

\bibitem{HJPolygons}
L.~Hillairet and C.~Judge,
\emph{Generic spectral simplicity of polygons},
Proc.\ Amer.\ Math.\ Soc.\ \textbf{137} (2009), 2139--2145.
\href{https://doi.org/10.1090/S0002-9939-09-09621-X}{doi:10.1090/S0002-9939-09-09621-X}.

\bibitem{Jimbo}
S.~Jimbo,
\emph{The singularly perturbed domain and the characterization for the eigenfunctions with Neumann boundary condition},
J.\ Differential Equations \textbf{77} (1989), 322--350.

\bibitem{Jones}
P.~W.~Jones,
\emph{Quasiconformal mappings and extendability of functions in Sobolev spaces},
Acta Math.\ \textbf{147} (1981), 71--88.

\bibitem{JudgeMondal}
C.~Judge and S.~Mondal,
\emph{Euclidean triangles have no hot spots},
Ann.\ of Math.\ (2) \textbf{191} (2020), 167--211;
\emph{Erratum: Euclidean triangles have no hot spots},
Ann.\ of Math.\ (2) \textbf{195} (2022), 337--362.
\href{https://doi.org/10.4007/annals.2020.191.1.3}{doi:10.4007/annals.2020.191.1.3};
\href{https://doi.org/10.4007/annals.2022.195.1.5}{doi:10.4007/annals.2022.195.1.5} (erratum).

\bibitem{KOT}
P.~Koskela, J.~Onninen, and J.~T.~Tyson,
\emph{Quasihyperbolic boundary conditions and Poincar\'e domains},
Math.\ Ann.\ \textbf{323} (2002), 811--830.
\href{https://doi.org/10.1007/s00208-002-0331-7}{doi:10.1007/s00208-002-0331-7}.

\bibitem{Murugan}
M.~Murugan,
\emph{Heat kernel for reflected diffusion and extension property on uniform domains},
arXiv:2304.03908v3 (2024).
\url{https://arxiv.org/abs/2304.03908}.

\bibitem{SmoothInput}
OpenAI,
\emph{Strict hot spots and absence of interior critical points on smooth simply connected planar domains},
OpenAI Math Release preprint, September 24, 2026.
\href{https://github.com/openai/math/blob/main/preprints/Strict-hot-spots-and-absence-of-interior-critical-points-on-smooth-simply-connected-planar-domains-September-24-2026/main.pdf}{Preprint dated September 24, 2026};
\href{https://github.com/openai/math/blob/main/lean/docs/369.md}{Lean scope note, item 369}.
Repository: \url{https://github.com/openai/math}. Accessed October 22, 2026.

\bibitem{PangII}
M.~M.~H.~Pang,
\emph{Stability and approximations of eigenvalues and eigenfunctions for the Neumann Laplacian, Part 2},
J.\ Math.\ Anal.\ Appl.\ \textbf{345} (2008), 485--499.
\href{https://doi.org/10.1016/j.jmaa.2008.04.026}{doi:10.1016/j.jmaa.2008.04.026}.

\bibitem{Siudeja}
B.~Siudeja,
\emph{Hot spots conjecture for a class of acute triangles},
Math.\ Z.\ \textbf{280} (2015), 783--806.
\href{https://doi.org/10.1007/s00209-015-1448-1}{doi:10.1007/s00209-015-1448-1}.

\end{thebibliography}
\end{document}